\documentclass[a4paper,12pt]{amsart}

\usepackage{graphicx}
\usepackage{caption}
\usepackage{subcaption}
\usepackage{pdfsync}
\usepackage[active]{srcltx}
\usepackage[english]{babel}
\usepackage{amscd}
\usepackage{amssymb}
\usepackage{amsthm}
\usepackage{amsmath}
\usepackage[latin2]{inputenc}

\usepackage{t1enc}
\usepackage{graphicx}
\usepackage{comment}
\usepackage{enumerate}
\usepackage{hyperref}
\usepackage{psfrag}
\usepackage{epsfig}
\usepackage{dsfont}
\usepackage{enumerate}
\usepackage{bbm}
\usepackage{caption}
\usepackage{geometry}
\usepackage{soul}
\usepackage[usenames,dvipsnames,svgnames]{xcolor}

\newcommand{\R}{\mathbb{R}}
\newcommand{\Q}{\mathbb{Q}}
\newcommand{\N}{\mathbb{N}}

\newcommand{\proj}{\mathrm{proj}}

\newcommand{\il}{\overline{\imath}}
\newcommand{\jl}{\overline{\jmath}}

\newtheorem*{condition*}{Condition}

\usepackage{wrapfig}
\usepackage{float}

\newcommand*{\e}[1]{\text{e}^{#1}}

\newcommand*{\ind}{\mathbbm{1}}

\newtheorem{theorem}{Theorem}[section]
\theoremstyle{plain}

\newtheorem{claim}{Claim}

\newtheorem{corollary}[theorem]{Corollary}

\newtheorem{definition}[theorem]{Definition}
\newtheorem{example}[theorem]{Example}

\newtheorem{lemma}[theorem]{Lemma}
\newtheorem{notation}[theorem]{Notation}

\newtheorem{proposition}[theorem]{Proposition}

\newtheorem*{assumptionA}{Assumption A}
\newtheorem*{assumptionB}{Assumption B}
\newtheorem*{assumptionC}{Assumption C}

\definecolor{trp}{rgb}{1,1,1}

\definecolor{red}{rgb}{1,0,.2}

\definecolor{blue}{rgb}{0,0,1}

\newcommand{\cblue}[1]{\clrblue{ #1}}

\definecolor{rgrey}{rgb}{.8,0.4,.4}  
\definecolor{grey}{rgb}{.13,.13,.13}  

\definecolor{green}{rgb}{0.0,0.4,0.2}

\newcommand{\bk}[1]{\cblue{ #1}\marginpar{\cblue{$\clubsuit$}}}

\usepackage{anysize}

\renewcommand{\bk}{}
\renewcommand{\cblue}{}

\begin{document}

\title[Triangular self-affine IFS]{On the dimension of triangular self-affine sets}


\author{Bal\'azs B\'ar\'any}
\address[Bal\'azs B\'ar\'any]{Budapest University of Technology and Economics, MTA-BME Stochastics Research Group, P.O.Box 91, 1521 Budapest, Hungary}
\email{balubsheep@gmail.com}

\author{Micha\l\ Rams}
\address[Micha\l\ Rams]{Institute of Mathematics, Polish Academy of Sciences, ul. \'Sniadeckich 8, 00-656 Warszawa, Poland}
\email{rams@impan.pl}

\author{K\'aroly Simon}
\address[K\'aroly Simon]{Budapest University of Technology and Economics, Department of Stochastics, Institute of Mathematics, 1521 Budapest, P.O.Box 91, Hungary} \email{simonk@math.bme.hu}

\subjclass[2010]{Primary 28A80 Secondary 28A78}
\keywords{Self-affine measures, self-affine sets, Hausdorff dimension.}
\thanks{The research of B\'ar\'any and Simon was partially supported by the grant OTKA K104745. B\'ar\'any was partially supported by the János Bolyai Research Scholarship of the Hungarian Academy of Sciences. Micha\l\ Rams was supported by National Science Centre grant 2014/13/B/ST1/01033 (Poland). This work was partially supported by the  grant  346300 for IMPAN from the Simons Foundation and the matching 2015-2019 Polish MNiSW fund. Part of Simon's research was supported by ICERM by supporting his participation on one of their a semester programs in 2016.}

\begin{abstract}
As a continuation of a recent work \cite{barany2016dimension} of the same authors, in this note we study the dimension theory of
diagonally homogeneous triangular planar self-affine IFS.
\end{abstract}

\date{\today}

\maketitle
\section{Introduction}
\subsection{The theme of the paper}
The dimension theory of self-affine measures and sets is so complicated that even on $\mathbb{R}^2$, in the diagonal case (when all the linear part of the mappings from the IFS are diagonal matrices),  it is not fully understood. The authors of this note have recently investigated this question \cite{barany2016dimension}. Namely,
consider the diagonal self-affine IFS on the plane
\begin{equation}\label{k53}
    \mathcal{S}^{\mathrm{diag}}\!:=\left\{S_i^{\mathrm{diag}}(x,y)\!:=
   D_i \!\cdot\! \left(
                                                \begin{array}{c}
                                                  x \\
                                                  y \\
                                                \end{array}
                                              \right)+
                                              \left(
                                                \begin{array}{c}
                                                  u_i \\
                                                  v_i \\
                                                \end{array}
                                              \right)
  \right\}
  _{i=1}^{N}\!, \mbox{where }
  D_i\!:=\left(
                                  \begin{array}{cc}
                                    c_i & 0 \\
                                    0 & b_i \\
                                  \end{array}
                                \right)\!.
\end{equation}
The projection of the coordinate axis,
 naturally generates a self-similar IFS on both coordinate axis. Assume that not both of the similarity dimensions of these projected self-similar IFS are greater than one.
In this case, the dimension theory of  diagonal self-affine systems on $\mathbb{R}^2$ are settled in \cite{barany2016dimension}, at least  for all but a very small set of parameters.

In this note we make a step forward and consider triangular self-affine IFS. That is we assume that the linear part of the mappings from the IFS are triangular matrices (all of them are lower triangular say). More precisely, let
\begin{equation}\label{k16}
  \mathcal{S}:=\left\{S_i(x,y):=T_i \cdot \left(
                                                \begin{array}{c}
                                                  x \\
                                                  y \\
                                                \end{array}
                                              \right)+
                                              \left(
                                                \begin{array}{c}
                                                  u_i \\
                                                  v_i \\
                                                \end{array}
                                              \right)
  \right\}
  _{i=1}^{N}, \mbox{where }
  T_i:=\left(
                                  \begin{array}{cc}
                                    c_i & 0 \\
                                    d_i & b_i \\
                                  \end{array}
                                \right).
\end{equation}
We say that $\mathcal{S}$
is diagonally homogeneous if the corresponding diagonal system $\mathcal{S}^{\mathrm{diag}}$ is homogeneous, that is, all the matrices $D_i$ in \eqref{k53} are identical. In this case we denote $c:=c_i$ and $b:=b_i$ for all $i$.
We mostly investigate the diagonally homogeneous case, see Section~\ref{k54} and Section~\ref{sec5}. However, in the general case, we have result in the case when affinity dimension is smaller than one, see Section~\ref{k60}.

\medskip

\subsection{History}
A self-affine Iterated Function System (IFS) is a finite list of contracting affine mappings on $\mathbb{R}^d$. If we choose a ball $B\subset \mathbb{R}^d$  centered at the origin with sufficiently high radius then this ball will be mapped into itself by all the mappings from the IFS.
The ellipses obtained by applying the mappings of the IFS, in any particular order $n$-times, on this ball $B$,
are the $n$-cylinders. As $n$ tends to infinity, the shapes of many of the $n$-cylinders become more and more relatively thinner and longer. This  makes it possible that even if the $n$-cylinders are pairwise disjoint, in some cases, they are not effective covers of the attractor (which is  the set that remains after infinite number of iterations of the mappings from the IFS on this ball $B$ above).

In 1988 Falconer introduced the notion of affinity dimension \cite{falconer1988hausdorff}
for self-affine fractals. We obtain it if we replace  the "most economic cover" in the definition of Hausdorff dimension with the most natural cover  associated with the  $n$-cylinders. In some sense the affinity dimension is the most natural guess for the Hausdorff dimension of a self-affine attractor.
Since the affinity dimension $\dim_{\rm {aff}}(\mathcal{F}) $ of a self-affine IFS
\begin{equation}\label{k55}
  \mathcal{F}:=\left\{f_i(x)=A_i \cdot \mathbf{x}+\mathbf{t}_i\right\}_{i=1}^{N}
\end{equation}
depends only on the linear parts, therefore it remains the same if we substitute $\mathcal{F}$ with its translations.
In 1988 Falconer proved that
 for almost all translates of a self-affine IFS, the Hausdorff dimension of the attractor and affinity dimension of the IFS  are equal, if
all of the mappings from the IFS has strong enough contraction. This upper bound on the contractions was originally $1/3$, which was improved 10 year later by Solomyak \cite{solomyak1998measure} to $1/2$.
For a survey of  results before 2014 see e.g. \cite{simon2014dimension}.
In the last two years there have been a very intensive development on this filed, partially due to the use of Furstenberg measure. See \cite{barany2012stationary}, \cite{barany2014ledrappier}, \cite{barany2015dimension}, \cite{barany2015ledrappier}, \cite{Rapaport}, \cite{FalconerKempton},  \cite{hochman2016dimension}.
\subsection{Affinity dimension in the triangular case on $\mathbb{R}^2$}
In the special case of the triangular self-affine IFS,
Falconer and Miao \cite[Corollary 2.6]{falconer2007dimensions} showed that if all the matrices are (e.g. lower) triangular then the affinity dimension can be given explicitly by a formula which depends only on the diagonal elements of the matrices. This formula is rather simple when we are on the plane. Namely, assume that the self-affine IFS is given in the formula
\eqref{k16} and let $s_x$ and $s_y$ be the similarity dimension of the self-similar IFS
\begin{equation}\label{k58}
 \mathcal{H}:= \left\{h_i(x):=c_ix+u_i\right\}_{i=1}^{N} \mbox{ and }
\mathcal{V}:=\left\{\varphi_i(x):=b_iy+v_i\right\}_{i=1}^{N}
\end{equation}
respectively. Clearly, $\mathcal{H}$ is the horizontal and $\mathcal{V}$ is the vertical
projection
 of the corresponding diagonal system given in the form
\eqref{k53}. Let
$$
\widehat{s}_x:=\min\left\{s_x,1\right\}
\mbox{ and }
\widehat{s}_y:=\min\left\{s_y,1\right\}
$$
We define $d_x$ and $d_y$ as the solutions of the following equations:
\begin{equation}\label{k56}
  \sum\limits_{i=1}^{N}c_i^{\widehat{s}_x}b_{i}^{d_x-\widehat{s}_x}=1
  \mbox{ and }
  \sum\limits_{i=1}^{N}b_i^{\widehat{s}_y}c_{i}^{d_y-\widehat{s}_y}=1.
  \end{equation}
Then (c.f. \cite[Theorem 4.1]{barany2012dimension}) the affinity dimension of $\mathcal{S}$ is
\begin{equation}\label{k57}
  \dim_{\rm{aff}}\left(\mathcal{S}\right)
  =
  \max\left\{d_x,d_y\right\}.
\end{equation}
 We say that direction-$x$, (direction-$y$)  dominates if $\dim_{\rm{aff}}\left(\mathcal{S}\right)=d_x$,
 ($\dim_{\rm{aff}}\left(\mathcal{S}\right)=d_y$) respectively. It follows from the definition of the Hausdorff dimension that the affinity dimension is always an upper bound for the Hausdorff dimension of the attractor (see \cite{falconer1988hausdorff}).

\subsection{Notation}\label{k92}

Throughout this note,  all self-affine IFS on the plane are supposed to be of the form of \eqref{k16}. Without loss of generality we may assume that
$$S_i([0,1]^2)\subset [0,1]^2 \mbox{ for all } i=1, \dots ,N$$

As we mentioned above, the attractor of $\mathcal{S}$ is $$\Lambda:=\bigcap\limits_{n=1}^{\infty }\bigcup_{\mathbf{i}\in\left\{1, \dots ,N\right\}^n}S_{\mathbf{i}}([0,1]^2),
$$
where $S_{\mathbf{i}} :=S_{i_1}\circ\cdots\circ S_{i_n} $ for $\mathbf{i}=(i_1, \dots ,i_n)$.
Let $\mu$ be the uniform distribution measure on the symbolic space $\Sigma:=\left\{1, \dots ,N\right\}^{\mathbb{N}}$. That is $\mu$ is the $(1/N,\ldots,1/N)$-Bernoulli measure.
We define $\Pi:\Sigma\to\Lambda_\mathcal{S}$ in the natural way.
$$
\Pi(\mathbf{i}):=\lim\limits_{n\to\infty} S_{\mathbf{i}|_n}(0),\quad \mathbf{i}\in\Sigma.
$$
 The push forward measure of $\mu$ is
  $\nu:=\Pi_*\mu.$ The attractor of the self-similar IFS $\mathcal{H}$ introduced in \eqref{k58}
is denoted by
	$\Lambda_\mathcal{H}$.
The natural projection generated by $\mathcal{H}$ is
$$
\Pi_{\mathcal{H}}(\mathbf{i}):=\lim\limits_{n\to\infty}
h_{\mathbf{i}|_n}(0),\quad \mathbf{i}\in\Sigma.
$$
Clearly, $\Pi_{\mathcal{H}}=\mathrm{proj}_x\circ \Pi$, where $\mathrm{proj}_x$ is the orthogonal projection to the $x$-axis. The measure on the $x$-axis generated by $\mu$ is
\begin{equation}\label{k93}
  \nu_x:=(\Pi_{\mathcal{H}})_*\mu=(\mathrm{proj}_x)_*\nu.
\end{equation}
Now we introduce the Furstenberg measure. The projection $\Pi_{AF}$ below will be used to construct a method to check that the transversality condition holds.

\subsection{Furstenberg measure}\label{k68}
In this subsection we study the action of our system $\mathcal{S}$ on the projective line, in the case
\begin{equation}\label{k70}
  c_i>b_i,\quad, \mbox{ for all }i=1, \dots ,N.
\end{equation}
In particular, the direction-$x$ dominates.

This action can be identified with the action of a simple iterated function system on the line.
Consider the vertical line   $\xi:=\left\{(1,z)\in\mathbb{R}^2:z\in \mathbb{R}\right\}$ on the plane.
We can identify $(1,z)\in\xi$ with $\widetilde{z}\in \mathbb{R}$. With this identification we define the self-similar IFS $\mathcal{F}$ on $\xi$ by
\begin{equation}\label{k69}
 \mathcal{F}:=\left\{f_i(\widetilde{z}):=\frac{b_i}{c_i}\widetilde{z}+\frac{d_i}{c_i}\right\}_{i=1}^{N},
\end{equation}
(Recall that in this Subsection $c_i>b_i$.)
It follows from \eqref{k70} that all $f_i$ are strict contractions.  So, we can define the natural projection
$\Pi_{AF}:\Sigma\to\xi$ in the usual way:
\begin{equation}\label{k71}
\Pi_{AF}(\mathbf{i}):=\frac{d_{i_1}}{c_{i_1}}+
\sum\limits_{k=2}^{\infty }\frac{d_{i_k}}{c_{i_k}} \cdot
\prod_{\ell =1}^{k-1}\frac{b_{i_{\ell} }}{c_{i_{\ell }}}.
\end{equation}
The importance of $\Pi_{AF}$ is as follows:
The action of $\left\{T_i\right\}_{i=1}^{N}$ on the projective line is described by the maps
 $\widetilde{T}_i:\xi\to\xi$
\begin{equation}\label{c33}
  \widetilde{T}_i(\widetilde{z}):=\frac{1}{c_i} \cdot T_i \cdot\left(
                                                                      \begin{array}{c}
                                                                        1 \\
                                                                        z \\
                                                                      \end{array}
                                                                    \right),
\end{equation}
where $\widetilde{z}\in\xi$ is $\widetilde{z}=(1,z)$.

Then
\begin{equation}\label{k72}
  \Pi_{AF}(\mathbf{i})=\widetilde{T}_i\left(\Pi_{AF}(\sigma\mathbf{i})\right)
\end{equation}
is the natural projection for $\left\{\widetilde{T}_i\right\}_{i=1}^{N}$.

As a similar construction have first appeared in \cite{FurstKif}, we will call the projection under $\Pi_{AF}$ of an ergodic measure $\eta$ defined on $\Sigma$ the {\it Furstenberg measure} corresponding to $\eta$.

\subsection{The transversality condition}

Consider two cylinders $S_{\mathbf{i}|_n}[0,1]^2$ and
$S_{\mathbf{j}|_n}[0,1]^2$, $i_1\ne j_1$. They are parallelograms with two  vertical sides. Their angle can be defined as the angle between their non-vertical sides.
The following condition holds if any two cylinders $S_{\mathbf{i}|_n}[0,1]^2$ and
$S_{\mathbf{j}|_n}[0,1]^2$, $i_1\ne j_1$ are either disjoint or have an angle uniformly separated from zero.
\begin{definition}\label{k59}
 We say that  $\mathcal{S}$ satisfies the transversality condition if is there exists a $c_3>0$ such that for every $n$ and for every $\mathbf{i},\mathbf{j}\in\Sigma$ with $i_1\ne j_1$
        we have
        \begin{equation}\label{k3}
          \mathrm{proj}_x\left(S_{\mathbf{i}|_n}[0,1]^2\cap
          S_{\mathbf{j}|_n}[0,1]^2
          \right)<c_3 \cdot b^n.
        \end{equation}
\end{definition}
Below we present a natural geometric interpretation of this condition which provides a method to check it.
First let us define the IFS $\widetilde{\mathcal{S}}$ which acts on $\mathbb{R}^3$
\begin{equation}\label{j36}
  \widetilde{\mathcal{S}}:=\left\{
\widetilde{S}_i(x,y,z):=\left(S_i(x,y),\widetilde{T}_i(z)
\right)\right\}_{i=1}^{N}.
\end{equation}

Now we recall two separation properties of IFS.
For an IFS
  $\{F_i\}$  we say that it satisfies the Strong Separation Property (SSP) if its natural projection (in case of $\widetilde{\mathcal{S}}$ it is given by $\mathbf{i}\mapsto(\Pi(\mathbf{i}),\Pi_{AF}(\mathbf{i}))$) is a bijection.
This is equivalent to the existence of a non-empty open set $V$ satisfying
\begin{description}
  \item[(a)] $F_i(V)\subset {V}$ for all $i=1, \dots ,N$ and
  \item[(b)] $\overline{F_i(V)}\cap \overline{F_j(V)}=\emptyset$ for all $i\ne j$,
\end{description}
where $\overline{A}$ means the closure of the set $A$.

We can define the Open Set Condition (OSC) in an analogous way: an IFS $\{F_i\}$  satisfies OSC if there exists a non-empty open set $V$ satisfying
\begin{description}
  \item[(a)] $F_i(V)\subset {V}$ for all $i=1, \dots ,N$ and
  \item[(b')] ${F_i(V)}\cap {F_j(V)}=\emptyset$ for all $i\ne j$,
\end{description}

\begin{lemma}\label{k73}\
  \begin{description}
    \item[(i)] If $\widetilde{\mathcal{S}}$ satisfies the SSP then the transversality condition holds for ${\mathcal{S}}$.
    \item[(ii)] If the transversality condition holds for ${\mathcal{S}}$  and $F_i([0,1]^2)\subset(0,1)^2$ for all $i=1,\dots,N$ then $\widetilde{\mathcal{S}}$ satisfies SSP.
  \end{description}
\end{lemma}


\begin{proof}First we prove part  \textbf{(i)}.
We are going to work with long and thin parallelograms. We will call {\it the principal axis} of a parallelogram the direction of its long side.

When $\mathbf{i}\mapsto(\Pi(\mathbf{i}),\Pi_{AF}(\mathbf{i}))$) is a bijection, the usual compactness argument shows that there exists $\ell>0$ such that for any two symbolic sequences $\mathbf{i}, \mathbf{j}\in \Sigma$ with $i_1\neq j_1$ either
\[
\mathrm{dist}
(\Pi(\mathbf{i}), \Pi(\mathbf{j})) > \ell
\]
or
\[
\mathrm{dist}(\Pi_{AF}(\mathbf{i}),\Pi_{AF}(\mathbf{j})) > \ell.
\]
Take some $N > 2(-\log \ell)/\min(-\log c,-\log (b/c))$.
In the first case, for $n>N$ the parallelograms $S_{\mathbf{i}|_n}[0,1]^2, S_{\mathbf{j}|_n}[0,1]^2$ do not intersect at all. In the second case, they intersect but transversally (the angle between their their principal axes is larger than $\ell/2$). In both cases \eqref{k3} holds. This proves the assertion \textbf{(i)}.

\medskip

Now we prove  part  \textbf{(ii)}.
Assume now that \eqref{k3} holds. By the assumption, there exists $\ell>0$ such that $\Lambda \in (\ell, 1-\ell)^2$. Fix some large $n$, to be defined later. Fix also some interval $I$ such that $\widetilde{T}_i(I)\subset I$ for all $i$.

Consider two words $\mathbf{i}|_n, \mathbf{j}|_n$ with different first digits $i_1\neq j_1$. Assume for the moment that the parallelograms $S_{\mathbf{i}|_n}[\ell/2,1-\ell/2]^2, S_{\mathbf{j}|_n}[\ell/2,1-\ell/2]^2$ intersect each other. Then the parallelogram $S_{\mathbf{i}|_n}[0,1]^2$ must intersect an internal point of parallelogram $S_{\mathbf{j}|_n}[0,1]^2$,
a point which is in distance at least $\ell b^n/2$ from the vertical boundary and in distance at least $\ell c^n/2$ from the horizontal (non-vertical) boundary of $S_{\mathbf{j}|_n}[0,1]^2$. Hence, if \eqref{k3} holds then the angle between the principal axes of the parallelograms is at least $c_4=\ell/2c_3$. In particular, if $n$ was so large that $(b/c)^n |I| < c_4/2$ then
\begin{equation} \label{someeq}
\widetilde{S}_{\mathbf{i}|_n}([\ell,1-\ell]^2\times I) \cap \widetilde{S}_{\mathbf{j}|_n}([\ell,1-\ell]^2\times I) = \emptyset.
\end{equation}

On the other hand, when the parallelograms $S_{\mathbf{i}|_n}[\ell/2,1-\ell/2]^2, S_{\mathbf{j}|_n}[\ell/2,1-\ell/2]^2$ do not intersect each other then \eqref{someeq} also holds. Hence, it holds for all pairs of words $\mathbf{i}|_n, \mathbf{j}|_n$ with different first digits. This implies that $\widetilde{S}$ satisfies SSP for the set
\[
V = \bigcup_{\mathbf{k}|_{n-1} \in \{1,\ldots,N\}^{n-1}} \widetilde{S}_{\mathbf{k}|_{n-1}}((\ell,1-\ell)^2\times I^o).
\]

This is complete the proof of  part \textbf{(ii)} of the assertion.
\end{proof}


\subsection{Lyapunov exponents, Lyapunov dimension, projection entropy and exponential separation condition}

\begin{notation}\label{k76}
 Let
$\Psi=\left\{\psi_i\right\}_{i=1}^{N}$ be strictly contracting IFS on $\mathbb{R}^d$.  Let $\Sigma=\left\{1,\dots,N\right\}^{\mathbb{N}}$ be the  symbolic space, $\sigma$ the  left-shift operator on $\Sigma$ and
the natural projection is
  $\Pi(i_1,i_2,\dots)=\lim_{n\rightarrow\infty}\psi_{i_1}\circ\cdots\circ\psi_{i_n}(\underline{0})$.
Let $\mathbf{p}:=(p_1, \dots ,p_N)$ be a probability vector.
  This generates a Bernoulli measure  on $\Sigma$, which is denoted by
$\mathbb{P}_{\mathbf{p}}:=\left\{p_1,\dots,p_N\right\}^{\mathbb{N}}$.
The corresponding self-similar measure is
\begin{equation}\label{k81}
  \nu_{\Psi,\mathbf{p}}:=\Pi_*\mathbb{P}_{\mathbf{p}}
\end{equation}
Finally, for every $n \geq 1$ we put
\begin{equation}\label{j98}
  \Sigma_n:=\left\{1, \dots ,N\right\}^n \mbox{ and }
  \Sigma_*:=\bigcup\limits_{n=1}^{\infty }\Sigma_n.
\end{equation}
\end{notation}

\subsubsection{Lyapunov exponents and Lyapunov dimension}
The general definition of Lyapunov exponents can be found e.g. in \cite{WaltersBook}.
However, in the special case of systems generated by lower triangular matrices like in \eqref{k16} the vertical direction is preserved by the system. Hence,
 the Lyapunov exponents can be expressed by the diagonal elements only \cite{falconer2007dimensions}.

The Lyapunov dimension of a self-affine measure was introduced in \cite[Definition 1.6]{barany2012stationary}. It is always an upper bound on the Hausdorff dimension of the measure.
In the special cases we consider in this note, we can write down a simpler formula for the Lyapunov dimension
of $\nu_{\Psi,\mathbf{p}}$:

\begin{example}

\begin{description}
  \item[(a)] If $\Psi$ is a self-similar IFS on the line
then $\nu_{\Psi,\mathbf{p}}$ has one Lyapunov exponent which is defined as
\begin{equation}\label{k74}
  \chi_{\Psi,\mathbf{p}}:=-\sum\limits_{i=1}^{N}p_i \cdot \log|\psi'(0)|.
\end{equation}
  \item[(b)] Assume that $\Psi=\mathcal{S}^{\mathrm{diag}}$ is a diagonally self-affine IFS on the plane defined as in \eqref{k53}.  Then its   horizontal and vertical
projections of   are  self-similar IFSs $\mathcal{H}$ and $\mathcal{V}$ on the line of the form \eqref{k58}. In this case, $\nu_{\Psi,\mathbf{p}}$  has two (not necessarily distinct)
 Lyapunov exponents: $\chi_{\mathcal{H},\mathbf{p}}$ and
$\chi_{\mathcal{V},\mathbf{p}}$. We say that direction-$x$ dominates if the following counter intuitive condition holds:
\begin{equation}\label{k89}
  \chi_{\mathcal{H},\mathbf{p}} \leq \chi_{\mathcal{V},\mathbf{p}}
\end{equation}
Assuming for example that direction-$x$ dominates, the Lyapunov dimension can be expressed as
\begin{equation}\label{k90}
  D(\nu_{\Psi,\mathbf{p}})=
\left\{
  \begin{array}{ll}
    \frac{h(\mathbb{P}_{\mathbf{p}})}{\chi_{\mathcal{H},\mathbf{p}}}, & \hbox{if ;
$h(\mathbb{P}_{\mathbf{p}}) \leq \chi_{\mathcal{H},\mathbf{p}}$} \\
  \frac{h(\mathbb{P}_{\mathbf{p}})+\chi_{\mathcal{V},\mathbf{p}}-\chi_{\mathcal{H},\mathbf{p}}}{\chi_{\mathcal{V},\mathbf{p}}}  , & \hbox{if; $\chi_{\mathcal{H},\mathbf{p}}<h(\mathbb{P}_{\mathbf{p}}) \leq
\chi_{\mathcal{H},\mathbf{p}}+\chi_{\mathcal{V},\mathbf{p}}
$}\\
  2\frac{h(\mathbb{P}_{\mathbf{p}})}{\chi_{\mathcal{H},\mathbf{p}}+\chi_{\mathcal{V},\mathbf{p}}} , & \hbox{otherwise.}
  \end{array}
\right.
\end{equation}

\item[(c)] Observe that the formulas \eqref{k90} depend only on the Lyapunov exponents and the entropy.
Hence,
in the general lower triangular case  when $\Psi=\mathcal{S}$   given by the formula \eqref{k16}, the Lyapunov dimension depends only on the diagonal elements of the matrices, so
the equation \eqref{k90} still holds.

\end{description}

\end{example}


\subsubsection{Projection entropy}\label{k75}
  We recall the definition
of the projective entropy, which was introduced by Feng and Hu \cite{feng2009dimension}.

Here we use  Notation \ref{k76}. That is
let $\Psi=\left\{\psi_i\right\}_{i=1}^N$ be a strictly contracting IFS on $\mathbb{R}^d$.  Moreover, let $\mathfrak{m}$ be a $\sigma$-invariant ergodic measure on $\Sigma$.
 Let $\mathcal{P}=\left\{[1],\dots,[N]\right\}$ be the partition of $\Sigma$, where $[i]=\left\{\mathbf{i}\in\Sigma:i_1=i\right\}$ and we write $\mathcal{B}$ for the Borel $\sigma$-algebra of $\mathbb{R}^d$.
Feng and Hu \cite[Definition~2.1]{feng2009dimension}
defined the {\em projection entropy of $\mathfrak{m}$ under $\Pi$ with respect to $\Psi$}
\begin{equation*}
h_{\Pi}(\mathfrak{m}):=H_{\mathfrak{m}}(\mathcal{P}\mid\sigma^{-1}\Pi^{-1}\mathcal{B})-
H_{\mathfrak{m}}(\mathcal{P}\mid\Pi^{-1}\mathcal{B}),
\end{equation*}
where $H_{\mathfrak{m}}(\xi\mid\eta)$ denotes the usual conditional entropy of $\xi$ given $\eta$.

\subsubsection{Hochman' exponential separation condition}
Hochman introduced the following Diophantine-type condition in \cite{hochman2012self}.

Let $\Psi=\left\{\psi_i\right\}_{i=1}^N$ be a self-similar IFS on $\mathbb{R}$.
Let $\Delta_n(\psi)$ be the minimum of $\Delta(\il,\jl)$ for distinct
 $\il,\jl\in\Sigma_n$, where
\[
\Delta(\il,\jl)=\left\{\begin{array}{cc}
             \infty & \psi_{\il}'(0)\neq \psi_{\jl}'(0) \\
             \left|\psi_{\il}(0)-\psi_{\jl}(0)\right| & \psi_{\il}'(0)=\psi_{\jl}'(0).
           \end{array}\right.
\]

\begin{condition*}
We say that the self-similar IFS $\psi$ satisfies \underline{Hochman's exponential separation condition} if there exists an $\varepsilon>0$ and an $n_k\uparrow \infty $ such that
\begin{equation}\label{k78}
  \Delta_{n_k}>\varepsilon^{n_k}.
\end{equation}
\end{condition*}

\bk{We remark that in our earlier paper \cite[p.2]{barany2016dimension} we stated this condition in an unnecessarily strict way. Namely, in \cite[Condition]{barany2016dimension} we required
that \eqref{k78} holds for all elements of the sequence $\left\{\Delta_n\right\}_{n=1}^{\infty }$, while it is enough  that  this inequality holds only on a subsequence as stated above. However, all the assertions of \cite{barany2016dimension} remain valid under the weaker condition \eqref{k78} since in \cite{barany2016dimension} we never used the condition $\Delta_n>\varepsilon^n$ directly, we used only the conclusions of Hochman's Theorems \cite{hochman2012self}.}
\section{Theorems of Hochman and Feng, Hu}

Hochman proved the following very important assertion in \cite[Theorem~1.1]{hochman2012self}.
\begin{theorem}[Hochman]\label{k79}
	Here we use Notation  \ref{k76}.
Let $\Psi=\left\{\psi_i\right\}_{i=1}^{N}$ be a self-similar IFS  on the real line. Assume that $\Psi$ satisfies the Hochman's exponential separation condition. Let $\mathbf{p}=(p_1, \dots ,p_N)$ be an arbitrary probability vector.

Then
\begin{equation}\label{k80}
  \dim_H\left(\nu_{\Psi,\mathbf{p}}\right)=\min\left\{1,
\frac{h(\mathbb{P_{\mathbf{p}}})}{\chi_{\Psi,\mathbf{p}}}\right\},
\end{equation}
 where $h(\mathbb{P_{\mathbf{p}}})=-\sum_{i=1}^Mp_i\log p_i$.
\end{theorem}
The ratio on the right hand side of \eqref{k80} is the similarity dimension of $\nu_{\Psi,\mathbf{p}}$. That is
\begin{equation}\label{k82}
  \dim_{\rm S}\left(\nu_{\Psi,\mathbf{p}}\right):=
 \frac{h(\mathbb{P_{\mathbf{p}}})}{\chi_{\Psi,\mathbf{p}}}.
\end{equation}

\begin{theorem}\cite[Theorem~2.8]{feng2009dimension}\label{k83}
Here we use Notation \ref{k76} and  notation from \bk{Section} \ref{k75}. Let $\Psi$ be a self-similar IFS on the line and $\mathbf{p}$ be a probability vector. Then
\begin{equation}\label{k77}
  \dim_H(\nu_{\Psi,\mathbf{p}})=\frac{h_{\Pi}(\mathbb{P_{\mathbf{p}}})}{\chi_{\Psi,\mathbf{p}}}.
\end{equation}
\end{theorem}
If we put Theorem \ref{k79} together with Theorem \ref{k83} we obtain

\begin{corollary}\label{k84}
  Let $\Psi$ be a self-similar IFS on the line which satisfies the Hochman's exponential separation condition. Then $h_{\Pi}(\mathbb{P_{\mathbf{p}}})=\min\{h(\mathbb{P_{\mathbf{p}}}),\chi_{\Psi,\mathbf{p}}\}$.
\end{corollary}

Now we consider the diagonally self-affine case on the plane.

\begin{theorem}[Feng Hu]\label{k85}
  Given the diagonally self-affine IFS like in \eqref{k53}. We assume that it satisfies the SSP .
Its horizontal and vertical parts (see \eqref{k58}) are denoted by $\mathcal{H}$ and $\mathcal{V}$.

Fix a  probability vector $\mathbf{p}=(p_1, \dots ,p_N)$. We consider the Lyapunov exponents
$\chi_{\mathcal{H},\mathbf{p}}$ and $\chi_{\mathcal{V},\mathbf{p}}$ as in \eqref{k74}.
Without loss of generality we assume that \underline{direction-$x$ dominates}, which means that
$$0<\chi_{\mathcal{H},\mathbf{p}} \leq \chi_{\mathcal{V},\mathbf{p}}$$ holds.
Then
\begin{equation}\label{k88}
  \dim_{\rm H} (\nu_{\mathcal{S}^{\mathrm{diag}}, \mathbf{p}})
=
\frac{h(\mathbb{P}_{\mathbf{p}})}{\chi_{\mathcal{V},\mathbf{p}}}
+
\left(1-\frac{\chi_{\mathcal{H},\mathbf{p}}}{\chi_{\mathcal{V},\mathbf{p}}}\right) \cdot
\dim_{\rm H} \left(\nu_{\mathcal{H},\mathbf{p}}\right).
\end{equation}
\end{theorem}
Observe that \bk{the} Hausdorff dimension of the measure
$\nu_{\mathcal{S}^{\mathrm{diag}}, \mathbf{p}}$ is equal to its Lyaponov dimension if and only if
$\dim_{\rm H} \left(\nu_{\mathcal{H},\mathbf{p}}\right)=1$.

\section{The trivial case and further notation}\label{k60}
\subsection{The trivial case}
\begin{lemma}\label{k64}
  Assume that direction-$x$ dominates and
$
  \sum\limits_{k=1}^{N}|c_i|<1.
$
Further, we assume that the exponential separation condition holds for the IFS  $\mathcal{H}$ defined in \eqref{k58}. Then $\dim_{\rm H} \Lambda_{\mathcal{S}}=s_x$.
\end{lemma}
The proof is immediate since in this case $\dim_{\rm {aff}}(\mathcal{S})=s_x $ which is an upper bound always. The lower estimate comes from Hochman Theorem: the dimension of the attractor of $\mathcal{H}$ is equal to $s_x$ and this attractor is a projection of $\Lambda$. Apart from the trivial case, we can get results only if we assume that
\begin{equation}\label{k65}
  c:=c_i\mbox{ and }b:=b_i \mbox{ holds for all } i.
\end{equation}

\subsection{Further notation}\label{k91}
We use the notation of Section \ref{k92}.

That is, from now on we always assume that the matrices $T_i$ in \eqref{k16} are of the form
\begin{equation}\label{k66}
  T_i=\left(
                                  \begin{array}{cc}
                                    c & 0 \\
                                    d_i & b \\
                                  \end{array}
                                \right).
\end{equation}
Clearly, if $c>b$ then direction-$x$ dominates and if $b>c$ then direction-$y$ dominates.

We use the notation of Section \ref{k92}, where we introduced the uniformly distributed Bernoulli measure $\mu$ on the symbolic space $\Sigma$, the natural projection $\Pi$, the projection of $\mu$ to the attractor $\Lambda\subset \mathbb{R}^2$ was called $\nu$ and the projection of $\nu$ to the $x$-axis was denoted by $\nu_x$.


The measures $\mu$ and $\nu$ can be disintegrated, according to the partitions
$$
\xi_a^s:=\{\mathbf{i}\in\Sigma:\Pi_{\mathcal{H}(\mathbf{i})}=a\}
\mbox{ and }
\xi_a:=\left\{(x,y):x=a, y\in[0,1]\right\}
$$
\begin{equation}\label{k1}
  \mu(A^s)=\int \alpha_a^\mu(A)d\mu_x(a), \quad
  \nu(A)=\int \alpha_a^\nu(A)d\nu_x(a),
\end{equation}
for any sets $A^s\subset \Sigma$ and $A\subset [0,1]^2$.
That is the probability measures
$\alpha_a^\mu$ and $\alpha_a^\nu$ are supported by
$\xi_a^s$ and $\xi_a$ respectively.

 Clearly,
\begin{equation}\label{k10}
  \mu_x=(\mathrm{proj}_x)_*\nu=\nu_x.
\end{equation}
For an arbitrary $a\in [0,1]$ we write $\Pi_a$ for the restriction of the natural projection $\Pi$
to $\xi_a^s$. That is $\Pi_a:\xi_a^s\cap \Sigma\to
\xi_a\cap \Lambda_{\mathcal{S}}$

Then we have and
\begin{equation}\label{k2}
  (\Pi_a)_*\alpha_{a}^{\mu}=\alpha_{a}^{\nu}.
\end{equation}

Now we introduce the vertical distance, vertical ball and and vertical neighborhood of a set:
  $$
  \mathrm{dist}_y((x_0,y_0),(x,y)):=
  \left\{
    \begin{array}{ll}
      |y-y_0|, & \hbox{if $x=x_0$;} \\
      \infty , & \hbox{otherwise},
    \end{array}
  \right.
  $$
  $$
  B_y((x_0,y_0),r):=\left\{(x,y):x=x_0 \mbox{ and }
  |y-y_0|<r
  \right\}.
  $$
  Further, let
  $$
  U_y(H,r):=\bigcup\limits_{(x_0,y_0)\in H}B_y((x_0,y_0),r).
  $$
  We define the vertical distance in the symbolic space as well:

$$
\mathrm{dist}_y^s(\mathbf{i}, \mathbf{j}) :=
\left\{
    \begin{array}{ll}
      b^{|\mathbf{i} \wedge\mathbf{j}|}, & \hbox{if $\Pi_{\mathcal{H}}(\mathbf{i})=\Pi_{\mathcal{H}}(\mathbf{j})$;} \\
      \infty , & \hbox{otherwise},
    \end{array}
  \right.
$$

$B_y^s(\mathbf{i},r)$ and $U_y^s(H^s, r)$ are then defined analogously to $B_y((x,y),r)$ and $U_y(H,r)$.


\section{Direction-$x$ dominates}\label{k54}
In this case by \eqref{k57},
\begin{equation}\label{k67}
  \dim_{\rm {aff}}\mathcal{S}=
  \left\{
    \begin{array}{ll}
      \frac{\log N}{-\log c}, & \hbox{if $Nc<1$;} \\
      1+\frac{\log (Nc)}{-\log b}, & \hbox{if $Nc>1$.}
    \end{array}
  \right.
\end{equation}
As we have seen above the $Nc<1$ case is obvious. So, from now on we may assume that $Nc>1$.

Consider the iterated function system $\mathcal{H}$. If $Nc>1$ then the affinity dimension of this system is larger than 1. For systems like that there are many results proving (under different assumptions) that their natural measure is not only absolutely continuous but also has $L^q$ density for some $q>1$. For example, for Bernoulli convolutions the natural measure has $L^2$ density for almost all parameters (\cite{solomyak1995random}), has $L^q$ density for some $q>1$ for all parameters except a dimension zero subset (\cite{shmerkin2014absolute}), and by the new preprint of Shmerkin it has $L^q$ density for all $q>1$ for all parameters except a dimension zero subset (\cite{shmerkin2016furstenberg}). It turns out that the knowledge of $L^q$ properties of the density of natural measure is quite useful.

We are going to use the following assumptions:

\begin{assumptionA}\label{k15}\ 


   (A1) $c> \frac{1}{N}$,

   (A2) $b<\frac{1}{N}$,

   (A3) $\nu_x\ll\mathcal{L}\mathrm{eb}$  with $L^q$ density, for some $q>1$,

   (A4) $\mathcal{S}$ satisfies transversality. (Definition \ref{k59}.)

\end{assumptionA}


When we replace (A3) with a stronger assumption (B3), we relax  our assumption about $b$:

\begin{assumptionB}\label{k51}\


  (B1) $c>\frac{1}{N}$

 (B2)  $b<c$.

   (B3) $\nu_x\ll\mathcal{L}\mathrm{eb}$  with $L^q$ density, for all $q>1$,

   (B4) $\mathcal{S}$ satisfies transversality.
\end{assumptionB}







If we consider the corresponding diagonal system (the one in
\eqref{k53}) then assumptions (A1, B1) say that the similarity dimension to the $x$-axis is greater  than one in both cases, while (A2) postulates that the projection to the vertical axis of the corresponding diagonal system has similarity dimension less than one.



Our new results when $c>b$ are as follows:
\begin{theorem}\label{k17}
Let $\mathcal{S}$ be a self-affine  IFS of the form \eqref{k16} satisfying  Assumption A. Then
\begin{equation}\label{k18}
  \dim_{\rm H} \left(\Lambda\right)=\dim_{\rm H} \left(\nu\right)=
1+\frac{\log(Nc)}{-\log b}.
\end{equation}
\end{theorem}

\begin{theorem}\label{j99}
Let $\mathcal{S}$ be a self-affine  IFS of the form \eqref{k16} satisfying  Assumption B. Then
\begin{equation}\label{k94}
  \dim_{\rm H} \left(\Lambda\right)=\dim_{\rm H} \left(\nu\right)=
\min\left(2, 1+\frac{\log(Nc)}{-\log b}\right).
\end{equation}
\end{theorem}

Observe that Assumption A guarantees that $1+\log(Nc)/-\log b <2$.


\subsection{The proof of Theorem \ref{k17}}
First we give the proof assuming a density assertion (Proposition \ref{k4}).
The proof of Proposition \ref{k4} is given in \bk{Sections} \ref{j80} and \ref{j79}.
\subsubsection{The proof of Theorem \ref{k17} modulo a density lemma}\label{k25}

\begin{proof}[Proof of Theorem \ref{k17}]
  Let $h_x(\mu)$ be the projection entropy which corresponds to the projection $\Pi_{\mathcal{H}}$ (that is, the entropy of $(\mathcal{H}, \Pi_{\mathcal{H}}(\mu))$).
It follows from \cite[Proposition 4.14]{feng2009dimension} that

\begin{equation}\label{k24}
  -\lim\limits_{n\to\infty}
\frac{\log \alpha_{\Pi_{\mathcal{H}}(\mathbf{i})}^{\mu}(B_y^s(\mathbf{i},b^n))}{n}
=
h(\mu)-h_{x}(\mu)  \mbox{ for $\mu$ a.a.} \mathbf{i}.
\end{equation}
Observe that
$h(\mu)=\log N$.
On the other hand,  by  Assumption (A3), it follows from Corollary \ref{k84} that we have
\begin{equation}\label{k23}
 h_{x}(\mu)=-\log c.
\end{equation}
Putting these together we obtain that

  \begin{equation}\label{k95}
  -\lim\limits_{n\to\infty}
\frac{\log \alpha_{\Pi_{\mathcal{H}}(\mathbf{i})}^{\mu}(B_y^s(\mathbf{i},b^n))}{n}
=
  \log (Nc) \mbox{ for $\mu$ a.a.} \mathbf{i}.
\end{equation}
Our aim below is to prove the corresponding statement for
$\alpha_{a_0}^{\nu}$:
 \begin{equation}\label{k26}
  -\lim\limits_{n\to\infty}
\frac{\log \alpha_{a_0}^{\nu}(B_y(\Pi_{a_0}(y_0),b^n))}{n}
=
  \log (Nc) \mbox{ for $\nu$ a.a.} (a_0,y_0).
\end{equation}
Namely, if  Proposition \ref{k4} holds then \eqref{k26} follows and this implies that   for $\mu_x$-a.a. $a_0\in[0,1]$
we have
\begin{equation}\label{k27}
  \dim_{\rm H} (\Lambda_{\mathcal{S}}\cap\xi_{a_0}) \geq
\frac{\log (Nc)}{-\log b} .
\end{equation}
Using that $\mu_x\sim \mathcal{L}_1$ we obtain that \eqref{k27}
holds for a set of Hausdorff dimension $1$ of $a_0$. Then
by \cite[Theorem 5.8]{falconer1986geometry} we obtain that
\begin{equation}\label{k28}
  \dim_{\rm H} \Lambda_\mathcal{S} \geq 1+\frac{\log (Nc)}{-\log b}.
\end{equation}
The opposite inequality is immediate since $1+\frac{\log (Nc)}{-\log b}$
is the affinity dimension of $\mathcal{S}$ which is always an upper bound on the Hausdorff dimension of the attractor. So, to complete the proof of Theorem \ref{k17} it is enough to verify Proposition \ref{k4} below.
\end{proof}

\subsubsection{Density in the symbolic space versus on the attractor}\label{j80}

\begin{proposition}\label{k4}
  For $\mu$-a.a. $\mathbf{i}$ we have
  \begin{equation}
\liminf\limits_{n\to\infty} \frac{1}{n} \log \alpha_{\Pi_{\mathcal{H}}(\mathbf{i})}^{\mu}
\left(B_y^s\left(\mathbf{i},b^m\right)\right)
=
\liminf\limits_{n\to\infty} \frac{1}{n} \log \alpha_{\Pi_{\mathcal{H}}(\mathbf{i})}^{\nu}
\left(B_y\left(\Pi(\mathbf{i}),b^m\right)\right).
  \end{equation}
  The same holds if we replace $\liminf$ with $\limsup$ on both sides.
\end{proposition}
We need the following notion:
\begin{definition}[Definition of $L$]
  Let $\mathbf{i}\in \Sigma$ and we write
  \begin{equation}\label{k9}
   Z_n(\mathbf{i}):=
  \left\{\mathbf{j}:\mathbf{j}|_n\ne\mathbf{i}|_n,\ \Pi_{\mathcal{H}}(\mathbf{j})=\Pi_{\mathcal{H}}(\mathbf{i})\right\}=
[\mathbf{i}|_n]^c\cap \Pi_{\mathcal{H}}^{-1}\left(\Pi_{\mathcal{H}}(\mathbf{i})\right),
\end{equation}
where $[\mathbf{i}|_n]:=\left\{\mathbf{j}\in\Sigma:\mathbf{i}|_n=\mathbf{j}|_n\right\}$, and for a set $A$, we write $A^c$ for the complement of  $A$.
  We define the function    $L$ which is the vertical distance from the closest point having a different first cylinder in its symbolic representation: 
  \begin{equation}\label{k33}
	 L(\mathbf{i}):=\min\left\{\mathrm{dist}_y(\mathbf{i}, \mathbf{j}); \mathbf{j} \in Z_1(\mathbf{i})
\right\}.
  \end{equation}

   \end{definition}

  We say that an $\mathbf{i}\in\Sigma$ is
  \textit{$\varepsilon$-good} if $\exists C=C(\mathbf{i},\varepsilon)>0$ such that
  for all $n \geq 0$ we have $L(\sigma^n\mathbf{i}) \geq C \cdot \e{-\varepsilon n}$.
Moreover,
$\mathbf{i}\in\Sigma$ is called  \textit{good} if it is $\varepsilon$-good for all $\varepsilon>0$. That is, we write
\begin{equation}\label{k31}
  G:=\left\{
\mathbf{i}\in\Sigma: \mathbf{i} \mbox{ is } \varepsilon \mbox{-good for all } \varepsilon>0
\right\}.
\end{equation}

The geometric meaning of $G$ is given by following observation.

  \begin{claim}\label{k6}
   Assume that $\mathbf{i}\in\Sigma$ is $\varepsilon$-good for some $C,\varepsilon>0$. Let $\pmb{\omega}\in\left\{1, \dots ,N\right\}^n$ such that $\pmb{\omega}\ne\mathbf{i}|_n$. Then
   \begin{equation}\label{k7}
      \Lambda_{\pmb{\omega}}\cap B_y\left(
   \Pi(\mathbf{i}),C \cdot b^n\e{-\varepsilon n}
   \right)=\emptyset.
   \end{equation}
  \end{claim}
\begin{proof}
  Let  $k:=|\pmb{\omega}\wedge\mathbf{i}|<n$. By assumption we have
\begin{equation}\label{k11}
  \mathrm{dist}_y\left(\Pi(\sigma^k\mathbf{i}),
\Lambda_{\pmb{\omega}_{k+1}}\right) \geq C \cdot \e{-\varepsilon k}.
\end{equation}
Using that $S_{\mathbf{i}|_k}=S_{\pmb{\omega}|_k}$,
$\Pi(\mathbf{i})=S_{\mathbf{i}|_k}(\Pi(\sigma^k\mathbf{i}))$ and
$S_{\pmb{\omega}|_k}(\Lambda_{\omega_{k+1}})=\Lambda_{\pmb{\omega}|_{k+1}}\supset \Lambda_{\pmb{\omega}}$
the statement follows from the fact that $S_{\mathbf{i}|_k}$ contracts in vertical direction by factor $b^k$.
\end{proof}

\begin{lemma}\label{k8}
  Let $\mathbf{i}\in G$ and $\Pi(\mathbf{i})=(a_0,y_0)$. Let $\eta^s$ be a measure on $\xi_{a_0}^s$ and we write $\eta:=(\Pi)_*\eta^s$. Then
\begin{equation}
\liminf\limits_{n\to\infty} \frac{1}{n} \log \eta^s
\left(B_y^s\left(\mathbf{i},b^n\right)\right)
=
\liminf\limits_{n\to\infty} \frac{1}{n} \log \eta
\left(B_y\left(\Pi(\mathbf{i}),b^n\right)\right).
  \end{equation}
  The equality stays true if we replace $\liminf$ by $\limsup$ on both sides.
\end{lemma}

\begin{proof}

Clearly, $\eta$ is supported on $\xi_{a_0}$.
 The direction "$ \leq $" follows from the fact that $\Pi$ restricted to $\xi_{a_0}^s$ is a Lipschitz map (with Lipschitz constant 1). This is true for every $a_0$, hence the inequality in this direction is true for all $\mathbf{i}\in\Sigma$.

Now we verify the "$ \geq $" part.
Fix an $\varepsilon>0$ an $n$
and an $\mathbf{i}\in \Sigma$ which is $\varepsilon$-good with a constant $C>0$.

As $\mathbf{i}$ is $\varepsilon$-good, by \eqref{k7},
\[
\eta(B_y(\Pi(\mathbf{i}), C \cdot b^n\e{-\varepsilon n})) \leq \eta^s (\Pi(\mathbf{i}), b^n).
\]
Passing with $n$ to infinity, we get

\[
\left(1 - \frac \varepsilon {\log b}\right)\liminf\limits_{n\to\infty} \frac{1}{n} \log \eta^s
\left(B_y^s\left(\mathbf{i},b^n\right)\right)
\geq
\liminf\limits_{n\to\infty} \frac{1}{n} \log \eta
\left(B_y\left(\Pi(\mathbf{i}),b^n\right)\right).
\]

As $\varepsilon>0$ can be chosen arbitrarily small, the assertion follows.
 \end{proof}

\subsubsection{$\mu$-almost every point is good}\label{j79}

Our goal is to prove
\begin{proposition}\label{k29}
Let $\mu$ be the uniform distribution on $\Sigma$  and $G\subset\Sigma$ be defined by \eqref{k31}. Then
  \begin{equation}\label{k30}
    \mu(G)=1.
  \end{equation}
\end{proposition}
Assuming this, we obtain
\begin{proof}[Proof of Proposition \ref{k4}]
 The proof immediately follows from
Lemma \ref{k8} and
Proposition \ref{k29}.
\end{proof}
Recall that, as we discussed above,  the proof of Theorem \ref{k17} follows from Proposition \ref{k4}. So the only thing left is to prove Proposition \ref{k29}. To do so we need the following Lemma.
First we introduce
$$
V_{\varepsilon,n}:=\left\{\mathbf{j}:
L(\mathbf{j})<\e{-\varepsilon n}
\right\}.
$$
\begin{lemma}\label{k35}
There exists a constant $c_4>0$ an $r\in(0,1)$ such that for all $n \geq 0$.
\begin{equation}\label{k36}
  \mu\left(V_{\varepsilon,n}\right)<c_4 \cdot r^n.
\end{equation}
\end{lemma}

\begin{proof}
Let $\ell :=\frac{\varepsilon}{-\log b} \cdot n$. Then $b^\ell =\e{-\varepsilon n}$. In the rest of this proof we always suppose that $$\pmb{\omega},\pmb{\tau}\in\Sigma_\ell,\quad \omega_1\ne \tau_1.$$
Recall that
 by assumption there exist a $q>1$ such that $M:=\int \varphi^q(t)dt<\infty $, where $\varphi(t)$ is the density function of $\nu_x$. It follows from Assumption (A2) that we can
choose a $\delta\in(0,1)$ such that
\begin{equation}\label{k45}
  Nb \cdot (Nc)^{2\delta}<1
\end{equation}
 and let
$$
\mathrm{Bad}^1_{\ell ,\delta}:=\left\{t\in[0,1]:
\varphi(t)>\left(Nc\right)^{\delta\ell }
\right\}
$$
and
$$
\mathrm{Bad}^2_{\ell ,\delta}:=
\left\{
t\in[0,1]:
\sum\limits_{|\pmb{\tau}|=\ell }\ind_{h_{\pmb{\tau}}([0,1])}
(t)
 >
(Nc)^{(1+\delta)\ell }
\right\}
$$

It follows from Corollary \ref{L11} of the Appendix
that
\begin{equation}\label{k42}
  \nu_x\left(
\mathrm{Bad}^2_{\ell ,\delta}
\right)
 \leq
c_6  \cdot (Nc)^{-\ell \delta(q-1)},
\end{equation}
where $c_6>0$ is a constant. Furthermore by Markov inequality
$$
M=\int\limits_{0}^{1}\varphi^q(t)dt
=\int\limits_{0}^{1}\varphi^{q-1}(t)d\nu_x(t) \geq
\nu_x(\mathrm{Bad}^1_{\ell ,\delta}) \cdot (Nc)^{(q-1)\delta\ell }.
$$
that is
\begin{equation}\label{k44}
  \nu_x\left(\mathrm{Bad}^1_{\ell ,\delta}\right)
 \leq
M \cdot (Nc)^{-\delta\ell (q-1)}
\end{equation}
Now we define


\begin{equation}\label{j91}
I_{\pmb{\omega},\pmb{\tau}}:=
\mathrm{proj}_x
\left(
S_{\pmb{\omega}}([0,1]^2)
\cap
U_y\left(S_{\pmb{\tau}}([0,1]^2),b^\ell \right)
\right)\setminus \mathrm{Bad}_{\ell ,\delta}^{2}.
\end{equation}

We define
\begin{equation}\label{j90}
  R_{\pmb{\omega},\pmb{\tau}}:=\left(I_{\pmb{\omega},\pmb{\tau}}
\times
[0,1]
\right)\cap
S_{\pmb{\omega}}([0,1]^2).
\end{equation}
That is $R_{\pmb{\omega},\pmb{\tau}}$ consist of those
 elements of $S_{\pmb{\omega}}$ which are "bad" because of $S_{\pmb{\tau}}([0,1]^2)$.
It follows from transversality (Definition \ref{k59} ) that
\begin{equation}\label{k37}
  |I_{\pmb{\omega},\pmb{\tau}}|<3c_3 \cdot b^\ell .
\end{equation}
\begin{figure}
  \centering
  \includegraphics[width=10cm]{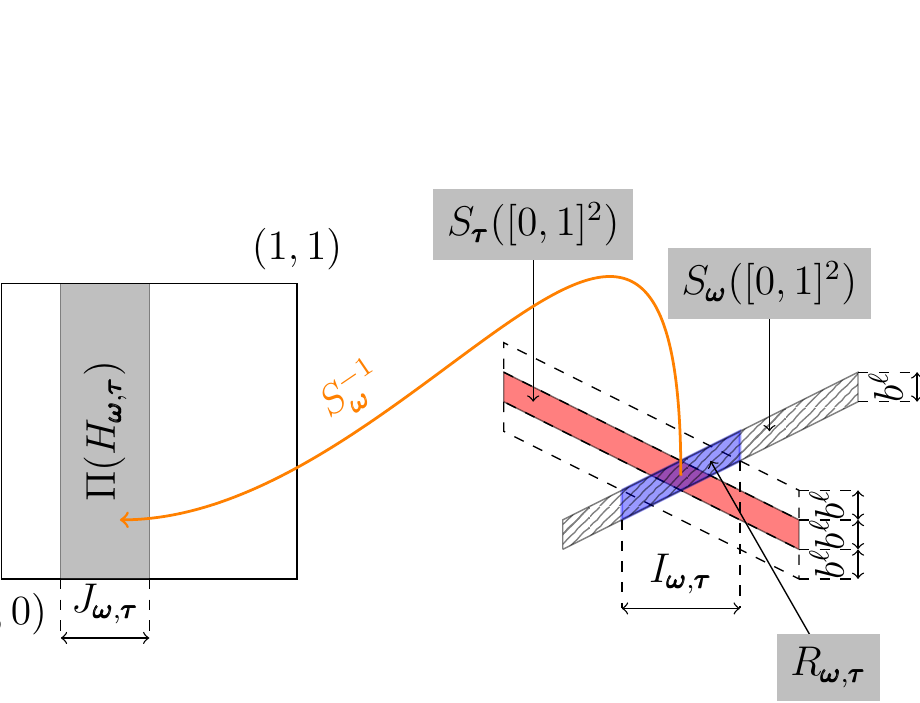}
  \caption{The idea of proof of Lemma \ref{k35}. ($\mathrm{Bad}_{\ell ,\delta}^{1}$, $\mathrm{Bad}_{\ell ,\delta}^{2}$ not included.)}
 \label{k50}
\end{figure}

Let
$$
\widetilde{R}_{\pmb{\omega},\pmb{\tau}}:=
[\pmb{\omega}]\cap\Pi^{-1}\left(
R_{\pmb{\omega},\pmb{\tau}}
\right)
.
$$
The importance of $\widetilde{R}_{\pmb{\omega},\pmb{\tau}}$ is that
it follows from elementary geometry that
\begin{equation}\label{k40}
  \left\{\mathbf{j}\in[\pmb{\omega}]:
\mathrm{proj}_x(\mathbf{j})\not\in \mathrm{Bad}_{\ell ,\delta}^{2},\
\exists \mathbf{i}\in\pmb{\tau},\
\mathrm{dist}_y(\Pi(\mathbf{i}),\Pi(\mathbf{j}))<b^\ell
\right\}\subset\widetilde{R}_{\pmb{\omega},\pmb{\tau}}
\end{equation}

It follows from the construction of $R_{\pmb{\omega},\pmb{\tau}}$ that there is a (possibly empty) interval $J_{\pmb{\omega},\pmb{\tau}}$
such that
$$
S_{\pmb{\omega}}^{-1}(R_{\pmb{\omega},\pmb{\tau}})=
(J_{\pmb{\omega},\pmb{\tau}}\setminus
h_{\pmb{\omega}}^{-1}(\mathrm{Bad}_{\ell ,\delta}^{2})
)\times [0,1],
$$
where the length of the  interval
\begin{equation}\label{k38}
  |J_{\pmb{\omega},\pmb{\tau}}|
 \leq 3c_3\left(\frac{b}{c}\right)^\ell .
\end{equation}
Let
$$
H_{\pmb{\omega},\pmb{\tau}}:=
\Pi^{-1}\left(
(J_{\pmb{\omega},\pmb{\tau}}\setminus
h_{\pmb{\omega}}^{-1}(\mathrm{Bad}_{\ell ,\delta}^{2}))\times [0,1]
\right).
$$
Hence the concatenation of $\pmb{\omega}$ and $H_{\pmb{\omega},\pmb{\tau}}$ is
\begin{equation}\label{k39}
 \widetilde{R}_{\pmb{\omega},\pmb{\tau}}=
\pmb{\omega} H_{\pmb{\omega},\pmb{\tau}}.
\end{equation}
To shorten the notation for an $\pmb{\omega}\in\Sigma_\ell $  we define
\begin{equation}\label{j93}
  D[\pmb{\omega}]:=\left\{\pmb{\tau}\in \Sigma_\ell  , \omega_1\ne\tau_1
\right\}.
\end{equation}

Using \eqref{k40} we can write
\begin{equation}\label{k41}
V_{\varepsilon,n}\subset
\left(\Pi_{\mathcal{H}}^{-1}\left(\mathrm{Bad}_{\ell ,\delta}^{2}\right)\right)\bigcup
\bigcup\limits_{\pmb{\omega}\in\Sigma_\ell  }\bigcup\limits_{\pmb{\tau}
\in D[\pmb{\omega}]}\widetilde{R}_{\pmb{\omega},\pmb{\tau}}.
\end{equation}
Hence, by \eqref{k42} and \eqref{k93} we get
\begin{eqnarray}\label{k43}
  \mu(V_{\varepsilon,n}) &=&
c_6(Nc)^{-\ell \delta(q-1)}
+
\sum\limits_{\pmb{\omega}\in\Sigma_\ell }
\mu\left(
\bigcup\limits_{\pmb{\tau}
\in D[\pmb{\omega}]}
\pmb{\omega}H_{\pmb{\omega},\pmb{\tau}}
\right)
\\
\nonumber&=&
c_6(Nc)^{-\ell \delta(q-1)}
+
   \frac{1}{N^\ell }\sum\limits_{\pmb{\omega}\in\Sigma_\ell }
\mu\left(
\bigcup\limits_{\pmb{\tau}
\in D[\pmb{\omega}]}
H_{\pmb{\omega},\pmb{\tau}}
\right)
\end{eqnarray}

To estimate $ \mu\left(
\bigcup\limits_{\pmb{\tau}
\in D[\pmb{\omega}]}
H_{\pmb{\omega},\pmb{\tau}}
\right)$ we present \bk{ }
$$
\bigcup\limits_{\pmb{\tau}
\in D[\pmb{\omega}]}
H_{\pmb{\omega},\pmb{\tau}}
=\cblue{\Pi^{-1}}\left(
\left(\mathrm{Bad}^1_{\ell ,\delta}
\cup
\bigcup\limits_{\pmb{\tau}
\in D[\pmb{\omega}]}\widetilde{J}_{\pmb{\omega},\pmb{\mathbf{\tau}}}\right)\times [0,1]\right),
$$
where
$$
\widetilde{J}_{\pmb{\omega},\pmb{\tau}}:=
J_{\pmb{\omega},\pmb{\tau}}\setminus
(h_{\pmb{\omega}}^{-1}(\mathrm{Bad}^2_{\ell ,\delta}\cup
\mathrm{Bad}^1_{\ell ,\delta}))
$$
It follows from the definition of the set $\mathrm{Bad}^1_{\ell ,\delta}$ that
\begin{equation}\label{k46}
\mbox{ if }t\in\widetilde{J}_{\pmb{\omega},\pmb{\tau}}\mbox{ then }
\varphi(t)<(Nc)^{\ell \delta}.
\end{equation}
On the other hand, it follows from the definition of $\mathrm{Bad}^2_{\ell ,\delta}$ that
\begin{equation}\label{k47}
  \sum\limits_{\tau \in D[\omega]}
  \ind_{\widetilde{J}_{\pmb{\omega},\pmb{\tau}}}(t) \leq (Nc)^{(1+\delta)\ell },\quad \forall t.
\end{equation}
Putting these together we get \bk{and using \eqref{k44} and \eqref{k38}}

\begin{eqnarray}
  \cblue{\mu}\left(
\bigcup\limits_{\pmb{\tau}
\in D[\pmb{\omega}]}
H_{\pmb{\omega},\pmb{\tau}}
\right)& \leq & \nu_x
\left(
\mathrm{Bad}^1_{\ell ,\delta}\right)
+
\nu_x\left(\bigcup\limits_{\pmb{\tau}\in D[\pmb{\omega}]} \widetilde{J}_{\pmb{\omega},\pmb{\tau}}\right)\\
 \nonumber   & \leq &
\nu_x\left(
\mathrm{Bad}^1_{\ell ,\delta}\right)
+
\int \sum\limits_{\tau \in D[\omega]}
  \ind_{\widetilde{J}_{\omega,\mathbf{\tau}}}(t)\varphi(t)dt
 \\
 \nonumber   & \leq & M \cdot (Nc)^{-\ell \delta(q-1)}
+(Nc)^{(1+\delta)\ell } (Nc)^{\ell \delta}3c_3\left(\frac{b}{c}\right)^\ell\\
 \nonumber   & \leq &
c_7\left(
(Nc)^{-\ell \delta(q-1)}+\left(Nb(Nc)^{2\delta}\right)^\ell
\right)
\end{eqnarray}

Putting this and \eqref{k43} together we obtain that
$$
\mu(V_{\varepsilon,n}) <
(c_6+c_7)
r^n,
$$
where
$$
r:=\max\left\{(Nc)^{-\delta(q-1)},Nb(Nc)^{2\delta}\right\}
^{\frac{\varepsilon}{-\log b}}.
$$

\end{proof}

\begin{corollary}\label{k48}
Let $R_1:=\left\{\mathbf{i}:\exists \mathbf{j}\ne \mathbf{i}, \ \Pi(\mathbf{i})=\Pi(\mathbf{j})\right\}$
Then
  $\mu\left(R_1\right)=0$.
\end{corollary}
\begin{proof}
It follows from Lemma \ref{k35} that
\begin{equation}\label{k49}
  \mu\left(\mathbf{i}: L(\mathbf{i})=0\right)=0.
\end{equation}
Clearly,
\begin{eqnarray*}
  R_1 &=& \bigcup\limits_{n=0}^{\infty }\left\{\mathbf{i}:
  \exists \mathbf{j},\ |\mathbf{i}\wedge\mathbf{j}|=n,\ \Pi(\mathbf{i})=\Pi(\mathbf{j}).
  \right\}\\
   &=&
   \bigcup\limits_{n=0}^{\infty }\sigma^{-n}\left\{\mathbf{i}:
  \exists \mathbf{j},\ i_1\ne j_1,\ \Pi(\mathbf{i})=\Pi(\mathbf{j}).
  \right\}
   \\
   &=& \bigcup\limits_{n=0}^{\infty }
  \sigma^{-n}\left\{\mathbf{i}:L(\mathbf{i})=0\right\}
\end{eqnarray*}
 This and \eqref{k49} imply that $\mu(R_1)=0$.
\end{proof}
The Proposition \ref{k29} is now immediate:
\begin{proof}[Proof of Proposition \ref{k29}]
$$
\left\{\mathbf{i}:
L(\sigma^{n}\mathbf{i})<\e{-\varepsilon n}
\right\}=\sigma^{-n}\left(V_{\varepsilon n}\right).
$$
From this and Lemma \ref{k35} we get that
$$
\sum\limits_{n}\mu
\left(\left\{\mathbf{i}:
L(\sigma^{n}\mathbf{i})<\e{-\varepsilon n}
\right\}\right)<\infty .
$$
Let $R_2:=\left\{\mathbf{i}:\exists N_0, \forall n>N_0\quad
L(\sigma^{n}\mathbf{i}) \geq \e{-\varepsilon n}
\right\}$. Then $\mu(R_2)=1$. It is immediate to see that every
$\mathbf{i}\in R_2\setminus R_1$ is $\varepsilon$-good. This means that the set $\varepsilon$-good $\mathbf{i}$ form a set of full measure. By taking a countable intersection we get that  $\mu(G)=1$.
\end{proof}

\subsection{$L^q \forall q$ density case}

In this section we will give the proof of Theorem \ref{j99}. In this subsection we assume that $\nu_x$ is absolute continuous with $L^q$ density for all $q>1$, we also assume the transversality condition.

In the previous subsection we proved Lemma \ref{k35}. In the calculation of $\mu(V_{\varepsilon, \ell })$ we were not able to make use of the following
fact: if the intervals
$\mathrm{proj}_x{S_{\pmb{\omega}}([0,1]^2)}$ and $\mathrm{proj}_x{S_{\pmb{\tau}}([0,1]^2)}$ intersect for some $\pmb{\omega},\pmb{\tau}\in\Sigma_\ell $,
it does not necessarily mean that the  parallelograms $S_{\pmb{\omega}}([0,1]^2)$ and $S_{\pmb{\tau}}([0,1]^2)$  intersect as well.
Under the $L^q \forall q$ assumption this distinction can be made.


\subsubsection{Number of pairs of intersecting cylinder parallelograms}\label{j68}
In this subsection we give an upper bound on the number of intersecting (or close-by) level $\ell $  cylinders with distinct first coordinates.
To state the lemma we need some preparation.

We write
\begin{equation}\label{j92}
  \Sigma_{\ell}^{2,\mathrm{diff}}:=\left\{(\pmb{\omega},\pmb{\tau})\in\Sigma_\ell \times\Sigma_\ell :\omega_1\ne\tau_1\right\}.
\end{equation}
For an $\pmb{\omega}\in\Sigma_\ell $ and $L>0$ let
\begin{equation}\label{j85}
  P_{\pmb{\omega}}^L:=
U_{y}\left(S_{\pmb{\omega}}([0,1]^2),Lb^\ell \right) \mbox{ and }
 I_{\pmb{\omega}}:=\mathrm{proj}_x{S_{\pmb{\omega}}([0,1]^2)}=\mathrm{proj}_x{P_{\pmb{\omega}}}
\end{equation}
Further, for $\pmb{\omega},\pmb{\tau}\in\Sigma_{\ell}^{2,\mathrm{diff}} $ we write
\begin{equation}\label{j86}
  P_{\pmb{\omega},\pmb{\tau}}^L:=P_{\pmb{\omega}}^L\cap P_{\pmb{\tau}}^L\supset
R_{\pmb{\omega},\pmb{\tau}},
\end{equation}
where $R_{\pmb{\omega},\pmb{\tau}}$ was defined in \eqref{j90}
and
\begin{equation}\label{j87}
 I^L_{\pmb{\omega,\pmb{\tau}}}:=\mathrm{proj}_x{P^L_{\pmb{\omega},\pmb{\tau}}}.
\end{equation}
Note that by transversality
\begin{equation}\label{j84}
  |I^L_{\pmb{\omega,\pmb{\tau}}}|<(L+2)c_3 \cdot
b^\ell .
\end{equation}

Let \bk{}
\begin{equation}\label{j61}
  B_{\cblue{\ell }}^L:=\#\left\{(\pmb{\omega},\pmb{\tau})\in
\Sigma_{\ell}^{2,\mathrm{diff}} :
P^L_{\pmb{\omega},\pmb{\tau}}\ne \emptyset
\right\}.
\end{equation}



\begin{lemma} \label{j97}Assume that Conditions (B1), (B3) and (B4) hold. 
Then for every $L>0$
\[
\limsup_{n\to\infty} \frac 1n \log B_n^L \leq \log (N^2c^2).
\]
\end{lemma}
\begin{proof}


	
Choose some $\widetilde{n}>1$. Let $n\geq \widetilde{n}$. Let us define a finite sequence $n_i$ as follows: $n_0:=n$ and
\begin{equation}\label{j96}
n_k:=\left\lfloor\left(\frac{\log c}{\log b}\right)^kn\right\rfloor.
\end{equation}
Let ${\mathcal{K}}:=\max\{k\geq0:n_k\geq \widetilde{n}\}$. Naturally, \bk{}
\[
B_{n_{\mathcal{K}}}^L \leq N^{{2}n_{\mathcal{K}}}.
\]
We show that for every $q>1$ there exists $K=K(q, L)$ such that for every $0\leq k\leq{\mathcal{K}}$
\begin{equation}\label{eq:ind}
B_{n_{k}}^L \leq B_{n_{k+1}}^L \cdot \left(K (Nc)^{n_{k}-n_{k+1}} c^{-(n_k-n_{k+1})/q}\right)^2.
\end{equation}

Let $\pmb{\omega},\pmb{\tau}\in\Sigma_{n_{k}}$. Assume that
\begin{equation}\label{j95}
  P_{\pmb{\omega}}^L\cap P_{\pmb{\tau}}^L\ne\emptyset.
\end{equation}
Observe that this is possible only if both
\begin{description}
  \item[(a)] $P^L_{\pmb{\omega}|_{n_{k+1}}}\cap P^L_{\pmb{\tau}|_{n_{k+1}}}\ne\emptyset$ and
  \item[(b)] $I_{\pmb{\omega}}\cap I^L_{\pmb{\omega}|_{n_{k+1}},\pmb{\tau}|_{n_{k+1}}}\ne\emptyset$ and
$I_{\pmb{\tau}}\cap I^L_{\pmb{\omega}|_{n_{k+1}},\pmb{\tau}|_{n_{k+1}}}\ne\emptyset$
\end{description}
hold.
By \eqref{j84} and  \eqref{j96}
$$
|I_{\pmb{\omega}|_{n_{k+1}},\pmb{\tau}|_{n_{k+1}}}| \leq
(L+2)c_3 b^{n_{k+1}} \approx c^{n_{k}}
$$
Hence by Corollary \ref{cor:add2}, for every $\pmb{\omega}|_{n_{k+1}}, \pmb{\tau}|_{n_{k+1}}$ there are at most $C(L, q) (Nc)^{n_{k}-n_{k+1}} c^{-(n_k-n_{k+1})/q}$ words $\pmb{\omega}\in\Sigma_{n_k}$
 such that
$I_{\pmb{\omega}}\cap I^L_{\pmb{\omega}|_{n_{k+1}},\pmb{\tau}|_{n_{k+1}}}\ne\emptyset$.
In the same way there are at most $C(L, q) (Nc)^{n_{k}-n_{k+1}} c^{-(n_k-n_{k+1})/q}$ words $\pmb{\tau}\in\Sigma_{n_k}$
 such that
$I_{\pmb{\tau}}\cap I^L_{\pmb{\omega}|_{n_{k+1}},\pmb{\tau}|_{n_{k+1}}}\ne\emptyset$.
 So, taking into consideration condition \textbf{(a)} above,
we obtain \eqref{eq:ind}.

Thus, by induction \bk{}
\[
B_{n_{0}}^L \leq (K)^{2\mathcal{K}} N^{\cblue{2}n_{\mathcal{K}}} (Nc)^{2(n_0-n_{\mathcal{K}})} c^{-2(n_0-n_{\mathcal K})/q}.
\]
But by definition of the sequence $\{n_k\}$, $n_0=n$, $n_{ \mathcal{K}}\approx \widetilde{n}$, and there exist constants $c_1,c_2\in\R$ such that $ \mathcal{K}\leq c_1\log n+c_2$. Therefore, 
$$
B_n^L\leq C(q,L) \cdot K(q,L)^{2c_1\log n} \cdot c^{-n/q} \cdot (Nc)^{2n} ,
$$
and passing with $q$ to infinity proves the assertion.
\end{proof}

\subsubsection{The corellation dimension}

 First we recall the definition of the correlation dimension (see \cite{Lau}).

\begin{definition}\label{j77}
  Let $\mathrm{m}$ be a positive bounded regular Borel measure on $\mathbb{R}^d$ with bounded support $\mathrm{spt}(\mathfrak{m})$. For every $r>0$ let $\left\{B_{i}^{(r)}\right\}_i$ be the $r$-mesh cubes that intersect $\mathrm{spt}(\mathfrak{m})$. For a $q>0$, $q\ne 1$ we define
  \begin{equation}\label{j76}
    \tau(q):=
    \liminf\limits_{r\to 0^+}
    \frac{\log \sum\limits_{i}\mathfrak{m}(B_{i}^{(r)})^q}{\log r}.
  \end{equation}
  Equivalently we could define $\tau(q)$ (see \cite{Lau}) in either of the following two ways:
  let
  \begin{equation}\label{j75}
    I_r(q):=\int\limits_{\mathbb{R}^d}
    \mathfrak{m}\left(B(x,r)\right)^qdx
    \mbox{ and }
    \widetilde{I}_r(q):=\int\limits_{\mathbb{R}^d}
    \mathfrak{m}\left(B(x,r)\right)^{q-1}d\mathfrak{m}(x).
  \end{equation}
  Then
  \begin{equation}\label{j74}
    \tau(q)
    =
  \liminf\limits_{r\to 0^+} \frac{\log I_r(q)}{\log r}
-d
    =
  \liminf\limits_{r\to 0^+} \frac{\log \widetilde{I}_r(q)}{\log r}
  -d.
  \end{equation}
  For $q>1$ the $L^q$-dimension of $\mathfrak{m}$ is defined as
  \begin{equation}\label{j73}
    \underline{\dim}_{\rm q}\left(\mathfrak{m}\right): =\frac{\tau(q)}{q-1}.
  \end{equation}
 \end{definition}
It was proved by Hunt and Kaloshin \cite[Proposition 2.1]{Hunt} that
\begin{align}
 \label{j72}  \underline{\dim}_{\rm q}\left(\mathfrak{m}\right) & =\sup\left\{t \geq 0:
 \int\left(\int
 \frac{d\mathfrak{\mathfrak{m}}(y)}{|x-y|^t}
 \right)^{q-1} d\mathfrak{m}(x)<\infty
 \right\} \\
 \nonumber  &= \inf\left\{t \geq 0:
 \int\left(\int
 \frac{d\mathfrak{\mathfrak{m}}(y)}{|x-y|^t}
 \right)^{q-1} d\mathfrak{m}(x)=\infty
 \right\}
 \\
 \nonumber  &=
 \sup\left\{s:
 \int\limits_{0}^{\infty }r^{-s(q-1)-1}\widetilde{I}_r(q)dr<\infty
 \right\}.
\end{align}
If we apply this for $q=2$ we get the correlation dimension of the measure $\mathfrak{m}$
$$
\dim_{\rm C}(\mathfrak{\mathfrak{m}}):= \underline{\dim}_{\rm 2}\left(\mathfrak{m}\right).
$$
It follows from \eqref{j72} that
\begin{equation}\label{j71}
  \dim_{\rm C}(\mathfrak{\mathfrak{m}}) \leq \dim_{\rm H} (\mathfrak{m}).
\end{equation}


\begin{proof}[Proof of Theorem \ref{j99}]
We will prove that under Assumption B
\[
\dim_C \nu \geq \min\left(2, 1 + \frac {\log (Nc)} {-\log b}\right).
\]

Let $\ell > 0$ and consider the grid of size $r=2b^\ell$. Let $B=[x-b^\ell, x+b^\ell]\times [y-b^\ell, y+b^\ell]$ be one of the $r$-mesh cubes. Let $|\pmb{\tau}|=\ell$ be such that $B$ intersects $S_{\pmb{\tau}}([0,1]^2)$. Then $S_{\pmb{\tau}}^{-1}(B\cap S_{\pmb{\tau}}([0,1]^2))$ is contained in some vertical strip of width $b^\ell/c^\ell$, hence by Corollary \ref{cor:add2}
\[
\nu(S_{\pmb{\tau}}^{-1}(B\cap S_{\pmb{\tau}}([0,1]^2))) \leq C(q) (b/c)^{\ell(1-1/q)}.
\]

Let $L:= (c-b)^{-1}\cdot \max\limits_{i \leq m}|d_i|$ be \bk{an upper bound on the } maximal slope of (the principal axis of) our cylinders. Then if the cylinder $S_{\pmb{\tau}}([0,1]^2), |\pmb{\tau}|=\ell$ intersects $B$, it must also intersect at least one of vertical intervals: either $\{x-b^\ell\} \times [y-(2L+1)b^\ell, y+(2L+1)b^\ell]$ or $\{x+b^\ell\} \times [y-(2L+1)b^\ell, y+(2L+1)b^\ell]$. Let us denote    the number of cylinders $S_{\pmb{\tau}}([0,1]^2)$ intersecting $\{x-b^\ell\} \times [y-(L+1)b^\ell, y+(L+1)b^\ell]$ by $Z_1(x,y)$
and the number of cylinders  intersecting $\{x+b^\ell\} \times [y-(L+1)b^\ell, y+(L+1)b^\ell]$ by $Z_2(x,y)$.
Using a similar estimate as  in the proof of Lemma \ref{k35}, in equation \eqref{k43}
we can write

\[
\nu(B) \leq C(q) N^{-\ell} \cdot (b/c)^{\ell(1-1/q)} \cdot (Z_1(x,y)+Z_2(x,y)),
\]
hence \bk{}
\[
(\nu(B))^2 \leq C(q)^{\cblue{2}} N^{-2\ell} \cdot (b/c)^{2\ell(1-1/q)} \cdot (2Z_1^2(x,y) + 2Z_2^2(x,y)).
\]
Thus, \bk{}
\[
\sum \nu(B_i^{(r)})^2 \leq C(q)^{\cblue{2}} N^{-2\ell} (b/c)^{2\ell(1-1/q)} \cdot \left(\sum Z_1^2(x_i, y_i) + \sum Z_2^2(x_i, y_i)\right),
\]
\cblue{where $B_i^{(r)}$ is the $i$-th $r$-mesh cube}. \bk{}
It is enough to estimate the first sum, the second is analogous. The interval $\{x-b^\ell\} \times [y-(2L+1)b^\ell, y+(2L+1)b^\ell]$ intersects $S_{\pmb{\tau}}([0,1]^2)$ if and only if the point $(x-b^\ell,y)$ is contained in $S_{\pmb{\tau}}([0,1]\times [-2L-1,2L+2])$. Hence, $Z_1^2(x, y)$ is equal to the number of pairs $\pmb{\omega},\pmb{\tau}\in\Sigma_\ell\times \Sigma_\ell$ such that $P_{\pmb{\omega},\pmb{\tau}}^{2L+1}$ contains $(x-b^\ell,y)$.

For every $\pmb{\omega},\pmb{\tau}\in\Sigma_\ell$ we write $k:=|\pmb{\omega}\wedge\pmb{\tau}|$. Then  for $\pmb{\alpha}:=\pmb{\omega}\wedge\pmb{\tau}\in \Sigma_k$, there exists
$\pmb{\beta},\pmb{\gamma}\in\Sigma_{\ell-k}$ such that
$$
\pmb{\omega}=\pmb{\alpha}\pmb{\beta}
\mbox{ and }
\pmb{\tau}=\pmb{\alpha}\pmb{\gamma}.
$$

\begin{figure}[H]
\vspace{-0.5cm}
  \centering
  \includegraphics[width=9cm]{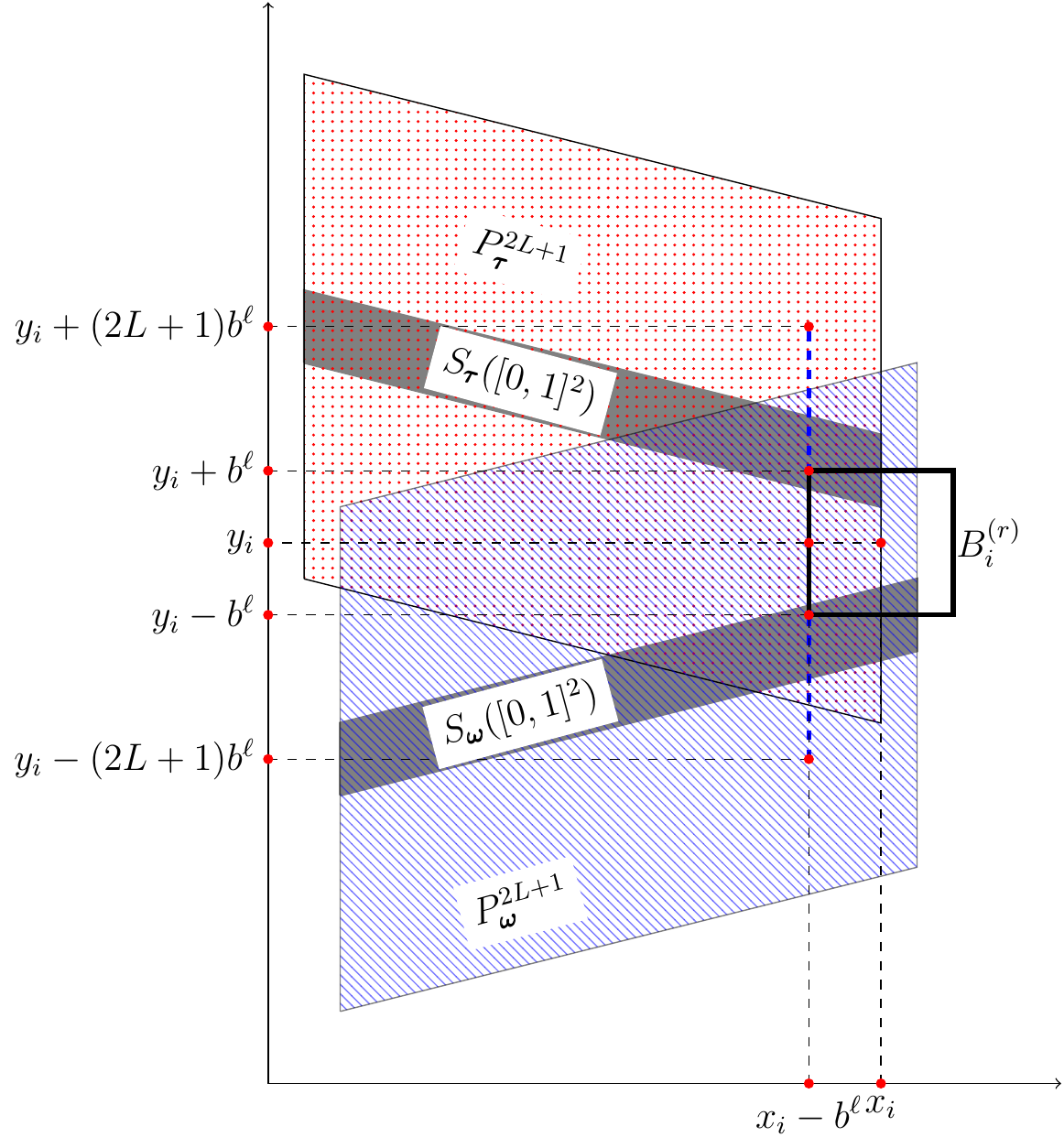}
  \caption{The region where (D4) holds.}\label{j35}
\end{figure}
With this notation we present this sum as

\[
\sum Z_1^2(x_i, y_i) = \sum_{k=0}^\ell  \sum_{\pmb{\alpha}\in\Sigma_k}
   \sum_{(\pmb{\beta},\pmb{\gamma})\in\Sigma_{\ell-k}^{2,\mathrm{diff}} } |\{i; (x_i-b^\ell, y_i) \in P_{\pmb{\omega},\pmb{\tau}}^{2L+1}\}|.
\]
Given $\pmb{\alpha}$, the last sum can be estimated by
\begin{equation}\label{k96}
  \sum_{(\pmb{\beta},\pmb{\gamma})\in\Sigma_{\ell-k}^{2,\mathrm{diff}} } |\{i; (x_i-b^\ell, y_i) \in P_{\pmb{\omega},\pmb{\tau}}^{2L+1}\}| \leq C B^{2L+1}_{\ell-k} \left(\frac cb \right)^k.
\end{equation}
Namely, by the Transversality Condition we can apply
\eqref{j84} which yields that \newline
 $\cblue{|I_{\pmb{\beta},\pmb{\gamma}}^{2L+1}| \leq (2L+5)c_3b^{\ell -k}}$ \cblue{and}
 $\cblue{|I_{\pmb{\omega},\pmb{\tau}}^{2L+1}|=c^k \cdot |I_{\pmb{\beta},\pmb{\gamma}}^{2L+1}|}$. Thus we have only $C \cdot (c/b)^k$ different $(x_i-b^\ell)$'s in \bk{} $\cblue{I}_{\pmb{\omega},\pmb{\tau}}^{2L+1}$ and for each of them we have at most $2L+2$ different $y_i$' in $P_{\pmb{\omega},\pmb{\tau}}^{L+1}$. Then \eqref{k96} follows from
 \eqref{j61}.

Applying Lemma \ref{j97} and noting that $\pmb{\alpha}$ can take $N^k$ values, we get

\[
\sum Z_1^2(x_i, y_i) \leq C N^{2\ell} c^{2\ell} o(N^{\ell \varepsilon}) \cdot \sum_{k=0}^\ell (Ncb)^{-k}.
\]
The last sum is a geometric series, hence it is bounded by a constant when $Ncb>1$ and by $(Ncb)^{-\ell}$ when $Ncb<1$. When $Ncb=1$ this sum equals $\ell +1 = o(N^{\ell \varepsilon})$. Thus,\bk{}

\[
\sum \nu(B_i^{(r)})^2 \leq C b^{\ell \cdot \min(2, 1+\log (Nc)/-\log b)} (b/c)^{-\cblue{2}\ell/q} o(N^{2\ell \varepsilon}).
\]
Passing with $q$ to infinity and $\varepsilon$ to 0 we get

\[
\tau(2) \geq \min \left(2, 1+\frac {\log (Nc)} {-\log b} \right).
\]
\end{proof}

\section{Direction-$y$ dominates}\label{sec5}

Finally, in this section we turn to the case when the direction-$y$ dominates. In this direction, we have only a mild development on the way of understanding the overlapping self-affine systems. The result can be considered as an extension of \cite[Theorem~B]{barany2016dimension} and \cite[Theorem~4.8, Theorem~4.11]{barany2014ledrappier}.

Similarly to the Section~\ref{k68}, we define the backward Furstenberg measure and IFS. This measure is supported on the directions, associated to the strong-stable directions. We note that in the case, when direction-$x$ dominates, the backward Furstenberg measure is supported on the singleton $\{(0,1)\}$.

Consider again the vertical line  $\xi:=\left\{(1,z)\in\mathbb{R}^2:z\in \mathbb{R}\right\}$ on the plane and  identify $(1,z)\in\xi$ with $\widetilde{z}\in \mathbb{R}$. Moreover, let $\mathcal{B}$ be the self-similar IFS on $\xi$ defined by
\begin{equation}\label{eq:backfurst}
\mathcal{B}:=\left\{g_i(\widetilde{z}):=\frac{c}{b}\widetilde{z}-\frac{d_i}{b}\right\}_{i=1}^{N},
\end{equation}
Let us define the natural projection by
$\Pi_{BF}:\Sigma\to\xi$ in the usual way:
\begin{equation}
\Pi_{BF}(\mathbf{i}):=-\frac{d_{i_1}}{b}-
\sum\limits_{k=2}^{\infty }\frac{d_{i_k}}{b} \cdot
\left(\frac{c}{b}\right)^{k-1}.
\end{equation}

Similarly, to \eqref{k72}, we have that the action of $\left\{T_i^{-1}\right\}_{i=1}^{N}$ on the projective line is described by the maps
$\widehat{T}_i:\xi\to\xi$
$$
\widehat{T}_i(\widetilde{z}):=c \cdot T_i^{-1} \cdot\left(
\begin{array}{c}
1 \\
z \\
\end{array}
\right),
$$
where $\widetilde{z}\in\xi$ is $\widetilde{z}=(1,z)$.

\begin{assumptionC}\label{assD}\ We assume that
	
	
	(C1) $c<\frac{1}{N}$,
	
	(C2) $b>c$
	
	(C3) The backward Furstenberg IFS $\mathcal{B}$ satisfies Hochman's exponential separation condition,

   (C4) $\mathcal{H}$ satisfies Hochman's exponential separation condition,
	\begin{flalign*}
	\phantom{=}\text{(C5) }\frac{\log N}{\log(b/c)}\geq\min\left\{1,\frac{\log N}{-\log b},2\left(1-\frac{\log N}{-\log c}\right)\right\}.&&
	\end{flalign*}
\end{assumptionC}

\begin{figure}[H]
\vspace{-0.5cm}
  \centering
  \includegraphics[width=9cm]{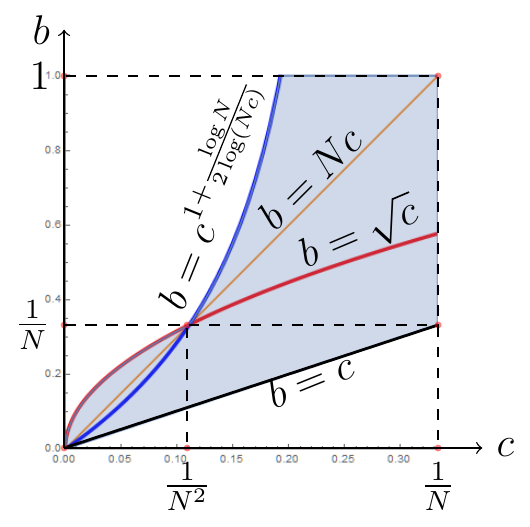}
  \caption{The region where (C5) holds.}\label{j35}
\end{figure}

Conditions (C1) and (C4) are devoted to be able to handle the projection entropy (defined in Section~\ref{k75}). Condition (C3) allows us to calculate the dimension of the backward Furstenberg measure and condition (C5) ensures that its dimension is larger than some possible exceptional set of orthogonal projections, for which the dimension drops.

\begin{theorem}\label{thm:c<b}
	Let $\mathcal{S}$ be a self-affine  IFS of the form \eqref{k16} satisfying  Assumption C. Then
	\begin{equation}\label{eq:thmy}
	\dim_{\rm H} \left(\Lambda\right)=\dim_{\rm H} \left(\nu\right)=
	\min\left\{\frac{\log N}{-\log b},1+\frac{\log(Nb)}{-\log c}\right\}.
	\end{equation}
\end{theorem}

As previously, let $\mu$ be the uniform Bernoulli measure on the symbolic space and let $\nu$ be its projection by the mapping $\Pi$, defined in Section~\ref{k92}.

For $\theta$ proper subspace of $\R^2$, let us denote the orthogonal projection from $\mathbb{R}^2$ to the subspace $\theta^{\bot}$ (orthogonal subspace to $\theta$) by $\mathrm{Proj}_{\theta}$. By \cite[Theorem~2.2]{barany2015ledrappier},
\begin{equation}\label{eq:lyformula}
\dim\nu=\frac{h_{\Pi}(\mu)}{-\log c}+\left(1-\frac{-\log  b}{-\log c}\right)\dim(\mathrm{Proj}_{\theta})_*\nu\text{ for }(\Pi_{BF})_*\mu\text{-a.e. }\theta,
\end{equation}
where $h_{\Pi}(\mu)$ is the projection entropy, defined in Section~\ref{k75}.

\begin{lemma}\label{lem:zerocondentropy}
	If (C1), (C2) and (C4) in Assumption C hold then $h_{\Pi}(\mu)=h(\mu)=\log N$.
\end{lemma}

\begin{proof}
	Let us define a lifted IFS on $[0,1]^3$ and a derived IFS on $\{0\}\times[0,1]^2$, as follows
	\[
	\widehat{\Phi}:=\left\{\widehat{S}_i(x,y,z)=(cx+u_{i},by+d_ix+v_{i},\rho z+w_i)\right\}_{i=1}^m\text{ and }
	\]
	\[
	\widetilde{\Phi}:=\left\{\widetilde{S}_i(y,z)=(cx+u_{i},\rho z+w_i)\right\}_{i=1}^m,
	\]
	where $0<\rho<\min\left\{|c|,|b|\right\}$ and $w_i\in\R$ are chosen such that
	\begin{equation}\label{eseplift}
	\widehat{S}_i([0,1]^3)\cap\widehat{S}_j([0,1]^3)=\emptyset\text{ and }\widetilde{S}_i([0,1]^2)\cap\widetilde{S}_j([0,1]^2)=\emptyset\text{ for every }i\neq j.
	\end{equation}
	Denote the natural projections of $\widehat{\Phi}$ and $\widetilde{\Phi}$ by $\widehat{\Pi}$ and $\widetilde{\Pi}$ respectively. Let us define $\widehat{\nu}=\widehat{\Pi}_*\mu$ and $\widetilde{\nu}=\widetilde{\Pi}_*\mu$ the push-down measures.
	
	We note that the Lyapunov exponents coincide for every measure $\widehat{\nu}, \widetilde{\nu}$, and $\nu$ for the appropriate directions. Applying \cite[Corollary~2.9]{barany2015ledrappier} and \cite[Theorem~2.11]{feng2009dimension}, we have
	\begin{eqnarray}
	\label{eq:ly2}\dim_H\widehat{\nu}&=&\frac{h(\mu)}{-\log\rho}+\left(\frac{-\log\rho+\log c}{-\log\rho}\right)\dim\nu+\\
	& &\left(\frac{-\log c+\log b}{-\log\rho}\right)\dim(\mathrm{Proj}_{\theta})_*\nu\text{ for }(\Pi_{BF})_*\mu\text{-a.e. }\theta,\\
	\dim_H\widetilde{\nu}&=&\frac{h_{\Pi_{\mathcal{H}}}(\mu)}{-\log c}+\frac{h(\mu)-h_{\Pi_{\mathcal{H}}}(\mu)}{-\log\rho}.
	\end{eqnarray}
	By \eqref{eq:lyformula} and \eqref{eq:ly2} we have
	\begin{equation}\label{eq1}
	\dim_H\widehat{\nu} = \frac{h(\mu)-h_{\Pi}(\mu)}{-\log\rho}+\dim\nu.
	\end{equation}
	
	Let us introduce measurable partitions of $[0,1]^3$ by $\xi(a,y):=\{a\}\times\{y\}\times[0,1]$ and $\tau(a):=\left\{a\right\}\times[0,1]\times[0,1]$. Moreover, define a measurable partition of $[0,1]\times\{0\}\times[0,1]$ by $\zeta(a)=\{a\}\times \{0\}\times[0,1]$ and a measurable partition of $[0,1]^2\times\{0\}$ by $\eta(a)=\{a\}\times[0,1]\times\{0\}$.

	By Rokhlin's Theorem there are families of conditional measures $\widehat{\nu}^{\xi}_{a,y}$, $\widehat{\nu}^{\tau}_a$, $\widetilde{\nu}^{\zeta}_a$ and $\nu^{\eta}_a$ on the partitions respectively, uniquely defined up to zero measure sets.
	
	By definition of conditional measures and the partition $\tau$, $\widehat{\nu}=\int\widehat{\nu}^{\tau}_ad\nu_x(a)$, where $\nu_x=(\Pi_{\mathcal{H}})_*\mu$, see \eqref{k10}. On the other hand,
	$\widehat{\nu}=\int\widehat{\nu}^{\xi}_{a,y}d\nu(a,y)=\iint\widehat{\nu}^{\xi}_{a,y}d\nu^{\eta}_a(y)d\nu_x(a)$. Thus, $$\widehat{\nu}^{\tau}_a=\int\widehat{\nu}^{\xi}_{a,y}d\nu^{\eta}_a(y)\text{ for $\nu_x$-a.e. $a$.}$$
	
	Let $\proj:[0,1]^3\mapsto\{0\}\times[0,1]^2$ be the orthogonal projection to the $y,z$-coordinate plane. Since $(\proj)_*\widehat{\nu}^{\tau}_a=\widetilde{\nu}_a^{\zeta}$ for $\nu_x$-a.e. $a$, we get that
	\begin{equation}\label{econdmeasures}
	\widetilde{\nu}_a^{\zeta}=\int(\proj)_*\widehat{\nu}^{\xi}_{a,y}d\nu^{\eta}_a(y)\text{ for $\nu_x$-a.e. $a$.}
	\end{equation}

	
	Applying \cite[Theorem~7.1]{barany2015ledrappier} we have
	\begin{eqnarray*}
		&&\dim\widehat{\nu}^{\xi}_{a,y}=\frac{h(\mu)-h_{\Pi}(\mu)}{-\log\rho}\text{ for $\nu$-a.e. $(a,y)$}\\
		&&\dim\widetilde{\nu}^{\zeta}_a=\frac{h(\mu)-h_{\Pi_{\mathcal{H}}}(\mu)}{-\log\rho}\text{ for $\nu$-a.e. $a$.}
	\end{eqnarray*}
	
	By assumptions (C1) and (C4), we may apply \cite[Theorem~2.8]{feng2009dimension} and \cite[Theorem~1.1]{hochman2012self}, and therefore
	\[
	\dim\nu_x=\frac{h_{\Pi_{\mathcal{H}}}(\mu)}{-\log c}=\frac{h(\mu)}{-\log c}.
	\]
	Therefore $\dim\widetilde{\nu}^{\zeta}_a=0$ for $\nu_x$-a.e. $a$.
	
	By \eqref{econdmeasures}, if $\widetilde{\nu}_a^{\zeta}(R)=0$ for a Borel set $R\subseteq\{a\}\times\{0\}\times[0,1]$ then $(\proj)_*\widehat{\nu}^{\xi}_{a,y}(R)=0$ for $\nu^{\eta}_a$-a.e $y$. Thus, by the definition of the Hausdorff dimension $\dim_H\widetilde{\nu}_a^{\zeta}\geq\dim_H(\proj)_*\widehat{\nu}^{\xi}_{a,y}=\dim_H\widehat{\nu}^{\xi}_{a,y}$ for $\nu$-a.e $(a,y)$. Hence $\dim_H\widehat{\nu}^{\xi}_{a,y}=0$ for $\nu$-a.e. $(a,y)$, which implies that $h(\mu)=h_{\Pi}(\mu)$.
\end{proof}

\begin{lemma}\label{lem:lowerbound}
	Let $\mathcal{S}$ be a self-affine  IFS of the form \eqref{k16} satisfying the assumptions of  Assumption C. Then
	
\begin{equation}\label{j32}
  \dim\nu\geq\min\left\{2\dfrac{h(\mu)}{-\log c},\dfrac{h(\mu)}{-\log b},1+\dfrac{h(\mu)+\log b}{-\log c}\right\}.
\end{equation}

\end{lemma}

\begin{figure}\label{j31}
  \centering
  \includegraphics[width=7cm]{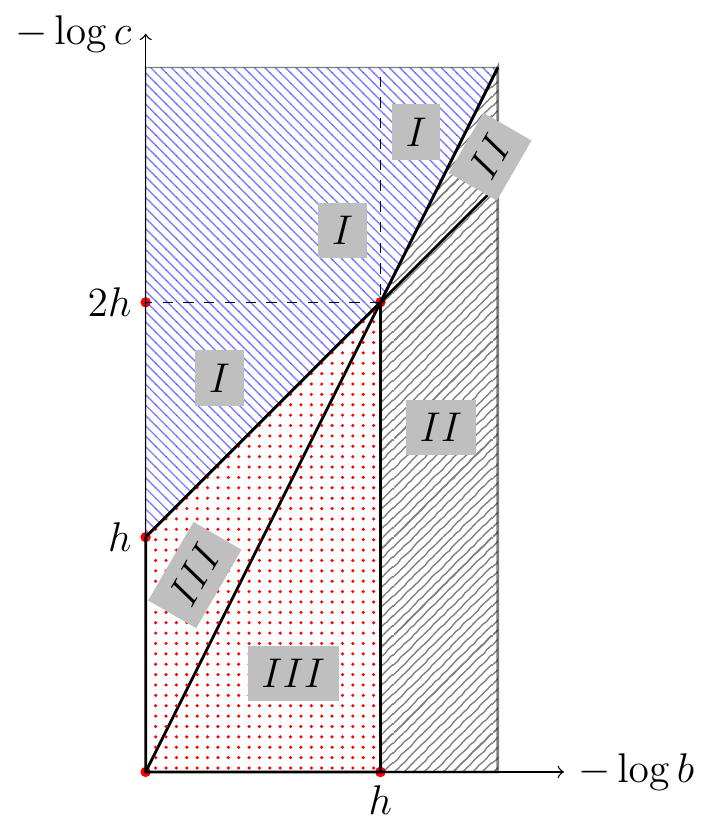}
  \caption{Region I,II and III corresponds to the  area where the minimum in \eqref{j32} is attained at the  first, second or third expression respectively.
  }\label{j30}
\end{figure}

\begin{proof}
	By \cite[Theorem~1.1]{hochman2012self},
	$$
	\dim(\Pi_{BF})_*\mu=\min\left\{1,\frac{h(\mu)}{\log b-\log c}\right\}.
	$$
	By \cite[Lemma~4.3]{barany2014ledrappier},
	\begin{equation}\label{eq:projection}
	\dim(\mathrm{Proj}_{\theta})_*\nu\geq\min\left\{\dim(\Pi_{BF})_*\mu,\dim\nu\right\}\text{ for }(\Pi_{BF})_*\mu\text{-a.e. }\theta.
	\end{equation}
	Let us define a sequence $\left\{x_n\right\}_{n=0}^{\infty}$ inductively as follows. Let $x_0=\dfrac{h(\mu)}{-\log c}$ and $x_n=r(x_{n-1})$ for $n\geq1$, where
	$$
	r(x)=\dfrac{h(\mu)}{-\log c}+\left(1-\dfrac{\log b}{\log c}\right)\min\left\{1,\dfrac{h(\mu)}{\log(b/c)},x\right\}.
	$$
	It is easy to see that
	$$
	\lim_{n\to\infty}x_n=\min\left\{2\dfrac{h(\mu)}{-\log c},\dfrac{h(\mu)}{-\log b},1+\dfrac{h(\mu)+\log b}{-\log c}\right\},
	$$
	which is the fixed point of $x\mapsto r(x)$. By applying \eqref{eq:lyformula} and Lemma~\ref{lem:zerocondentropy}, one can show by induction that $\dim\mu\geq x_n$ for every $n\geq0$, as required.
\end{proof}

\begin{proof}[Proof of Theorem~\ref{thm:c<b}]
	By \cite[Proposition~6.1]{PeresSchlag} and \eqref{eq:projection}, if
	\begin{equation}\label{eq:pseq}
	\dim(\Pi_{BF})_*\mu>\min\left\{\dim\nu,2-\dim\nu\right\}
	\end{equation}
	then
	\begin{equation}\label{eq:dimcons}
	\dim(\mathrm{Proj}_{\theta})_*\nu=\min\{1,\dim\nu\}\text{ for }(\Pi_{BF})_*\mu\text{-a.e. }\theta.
	\end{equation}
	But by \cite[Theorem~1.1]{hochman2012self}, $$
	\dim(\Pi_{BF})_*\mu=\min\left\{1,\frac{h(\mu)}{\log b-\log c}\right\}.
	$$ and by Lemma~\ref{lem:lowerbound},
	$$
	\min\left\{2\dfrac{h(\mu)}{-\log c},\dfrac{h(\mu)}{-\log b},1+\dfrac{h(\mu)+\log b}{-\log c}\right\}\leq\dim\nu.
	$$
	which together with assumtion (C5) implies \eqref{eq:pseq}. Thus, \eqref{eq:lyformula}, Lemma~\ref{lem:zerocondentropy} and \eqref{eq:dimcons} verify the statement.
\end{proof}

\section{Examples}

\subsection{Examples for the direction-$x$ dominates case}

\begin{figure}[H]\label{cd32}
    \centering
    \begin{subfigure}[b]{0.5\textwidth}
        \centering
        \includegraphics[width=\textwidth]{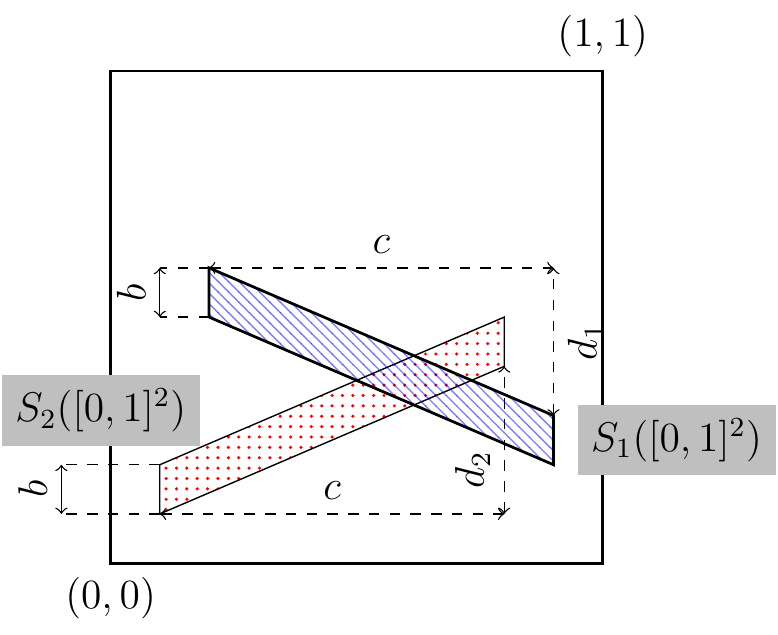}
        \caption{For Example \ref{j49}}\label{j47}
    \end{subfigure}%
    ~
    \begin{subfigure}[b]{0.5\textwidth}
        \centering
        \includegraphics[width=\textwidth]{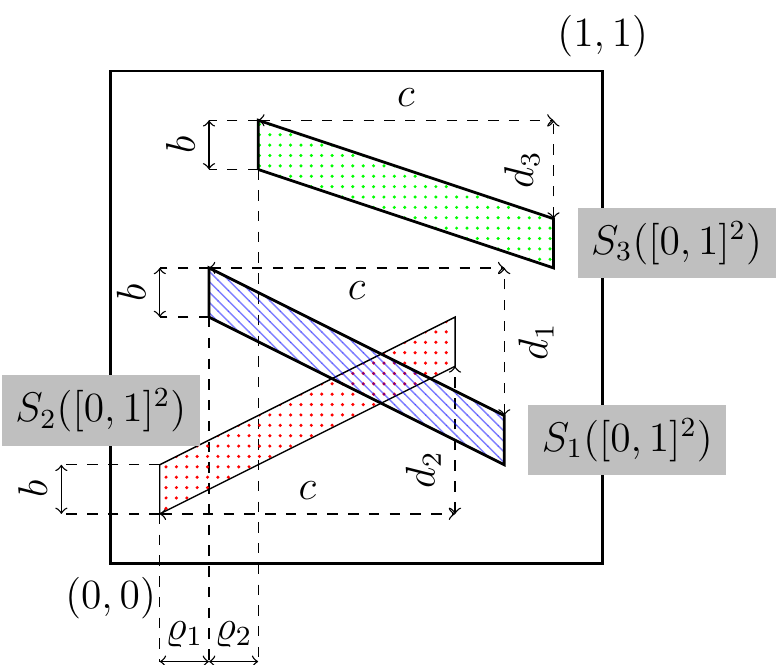}
        \caption{Example \ref{j29} and  \ref{j48}, in the Example \ref{j29}, $\rho_1=\rho_2$}\label{kep}
    \end{subfigure}
    \caption{ Negative values of $d_i$ correspond to "decreasing" parallelograms. In }
\end{figure}

We present three examples. In the first and second case we can apply Theorem \ref{k17} and in the third example we can apply Theorem~\ref{j99}. In all examples $c$ can be chosen as an arbitrary element of a parameter interval except a small exceptional set which is going to be $E_1\subset \left(0.5,1\right)$, $E_2\subset(1/3,1)$ and $E_3\subset \left(\frac{1}{\sqrt{3}},1\right)$ in the examples respectively with the properties:
\begin{equation}\label{j43}
  \dim_{\rm H} (E_1)=\dim_{\rm H} (E_2)=0\mbox{ and } \mathcal{L}\mathrm{eb}\left(E_3\right)=0.
\end{equation}
The precise definition of these exceptional sets are given in Section \ref{j39}.

\begin{example}\label{j49}
Let $\mathcal{S}$ be an IFS of the form \eqref{k16}, where
\begin{itemize}
  \item $N=2$.
  \item We choose an arbitrary $c\in\left(\frac{1}{2},1\right)\setminus E_1$ (see \eqref{j38}) and $0<b<c/2$.
  \item We assume that
$u_1\ne u_2$,
$d_1\ne d_2$ and $v_1,v_2$ are selected in such a way  that the IFS $\mathcal{S}:=\left\{S_1,S_2\right\}$
satisfies $S_i\left([0,1]^2\right)\subset [0,1]^2$.  Then by Theorem \ref{k17} we have
\begin{equation}\label{j41}
  \dim_{\rm H} (\Lambda)=
1+\frac{\log(2c)}{-\log b}.
\end{equation}
\end{itemize}
(See  Figure \ref{j47}.)

\end{example}
We can apply Theorem \ref{k17} in Example \ref{j49} because Assumption A holds. Namely, it is obvious that assumptions (A1)-(A2) hold. Assumption (A3)
follows from choice of $E_1$ (see \eqref{j38}).
Assumption (A4), (the transversality condition)  follows from the fact that the
associated $3$-dimensional IFS $\widetilde{\mathcal{S}}$, defined in \eqref{j36}, satisfies SSP.
 Namely, the third coordinate of $\widetilde{\mathcal{S}}$ is an IFS  $\mathcal{F}$ on the line, defined in  \eqref{k69}. This consists of two maps with distinct fixed points and the sum of their contraction ratios is less than $1$. This means that the SSP holds for $\mathcal{F}$. Consequently, the SSP holds for $\widetilde{\mathcal{S}}$. Hence the transversality condition (A4) also holds.
 \begin{example}\label{j29}
\begin{description}
  \item[(a)]  $N=3$.
  \item[(b)] We fix an arbitrary $c\in\left(\frac{1}{3},1\right)\setminus E_2$. (See \eqref{j37}.)
  \item[(c)] Let $b\in\left(0,\frac{c}{2}\right)\cap\left(0,\frac{1}{3}\right)$.
  \item[(d)] $u_1,u_2,u_3$  are pairwise different.
  \item[(e)] We assume that $d_1<d_2$ and we choose $d_3$ from the interval
$$
d_3\in \left(
d_2\left(2-\frac{c}{b}\right)+d_1\left(\frac{c}{b}-1\right)
,
d_1\left(2-\frac{c}{b}\right)+d_2\left(\frac{c}{b}-1\right)
\right).
$$
  Note that this holds for example if $d_1<d_3<d_2$.
  \item[(f)] We choose the vertical translation parameters (like on the right hand side of Figure \ref{kep}) so that
  \begin{equation}\label{j42}
    S_3\left([0,1]^2\right)\cap\left(S_1\left([0,1]^2\right)\cup S_2\left([0,1]^2\right)\right)=\emptyset,
  \end{equation}
  and
  $S_i\left([0,1]^2\right)\subset [0,1]^2$ for $i=1,2,3$.
\end{description}
  Then we have
\begin{equation}\label{j40}
\dim_{\rm H} (\Lambda)=1+\frac{\log(3c)}{-\log b}.
\end{equation}
(See  Figure \ref{kep}.)
\end{example}

\begin{example}\label{j48}
\begin{description}
  \item[(a)] $N=3$.
  \item[(b)] We fix an arbitrary $c\in\left(\frac{1}{\sqrt{3}},1\right)\setminus E_3$. (See \eqref{j37}.)
  \item[(c)] Let $b\in\left(0,\frac{c}{2}\right)$.
  \item[(d)] $u_1,u_2,u_3$  are consecutive elements of an arithmetic progression  (not necessarily in this order).
  \item[(e)] We assume that $d_1<d_2$ and we choose $d_3$ from the interval
$$
d_3\in \left(
d_2\left(2-\frac{c}{b}\right)+d_1\left(\frac{c}{b}-1\right)
,
d_1\left(2-\frac{c}{b}\right)+d_2\left(\frac{c}{b}-1\right)
\right).
$$
  Note that this holds for example if $d_1<d_3<d_2$.
  \item[(f)] We choose the vertical translation parameters (like on the right hand side of Figure \ref{kep}) so that
  \begin{equation}\label{j42}
    S_3\left([0,1]^2\right)\cap\left(S_1\left([0,1]^2\right)\cup S_2\left([0,1]^2\right)\right)=\emptyset,
  \end{equation}
  and
  $S_i\left([0,1]^2\right)\subset [0,1]^2$ for $i=1,2,3$.
\end{description}
  Then we have
\begin{equation}\label{j40}
\dim_{\rm H} (\Lambda)=\min\left\{
2,1+\frac{\log(3c)}{-\log b}
\right\}.
\end{equation}
(See  Figure \ref{kep}.)
\end{example}
Combining (c), (e) and (f) (note that these assumptions appear both in Examples \ref{j29} and \ref{j48}) we get that the transversality condition holds.
Namely, with the open set
$$
V:=
(0,1)^2\times
\left(\frac{\max\left(d_2,d_3\right)}{b-c},\frac{\min(d_1,d_3)}{b-c}\right)
$$
the IFS $\widetilde{\mathcal{S}}$ satisfies the SSP. Hence by Lemma \ref{k73} the transversality condition holds.
Hence, we can apply Theorem \ref{k17} for Example \ref{j29} in the same way we applied it for the Example \ref{j49}.

For the example \ref{j48} we are going to apply Theorem \ref{j99} here.
The combination of (b)
and Corollary \ref{j45} implies that $\nu_x$ has bounded density.
Hence,  Theorem \ref{j99} applies.

\subsubsection{Phase transition in Example \ref{j48}}
Note that in Example \ref{j48} the parameter interval $I=\left(\frac{1}{\sqrt{3}},1\right)$ can be partitioned naturally into
$$
I_1:=\left(\frac{1}{\sqrt{3}},\sqrt{\frac{2}{3}}\right]
\mbox{ and }
I_2:=\left(\sqrt{\frac{2}{3}},1\right].
$$
Namely, the affinity dimension
$$
A(b,c):=1+\frac{\log(3c)}{-\log b}
$$ is  monotone increasing in both $b$ and $c$. So, for a fixed $c$ it takes its biggest value for   $b=c/2$.
Clearly $c=\sqrt{2/3}$ is the solution of
$
A\left(\frac{c}{2},c\right)=2.
$
That is for $c\in I_1\setminus E_3$ we have $\dim_{\rm H} (\Lambda)=A(b,c)$ for every $0<b \leq c/2$. However, for $c\in I_2\setminus E_3$ it depends on the choice of $b$
if $\dim_{\rm H} (\Lambda)=A(b,c)$ or $\dim_{\rm H} (\Lambda)=2$. Now we fix  a
$c\in I_2\setminus E_3$, $u_1,u_2,u_3$ and $d_1,d_2,d_3$ satisfying the conditions of Example \ref{j48}. We vary only $b$. So, the attractor will be denoted by $\Lambda_b$.
To study the function
\begin{equation}\label{c36}
b\mapsto \dim_{\rm H} (\Lambda_b), \qquad b\in \left[0,\frac{c}{2}\right],
\end{equation}
first observe that
 for a $c\in I_2$ the function $ A\left(\cdot,c\right)$ is monotone increasing and
\begin{equation}\label{c35}
A\left(\frac{1}{3c},c\right)=2.
\end{equation}
That by Theorems \ref{j99}, for a $c\in I_2\setminus E_3$ we have:
\begin{equation}\label{c34}
  \dim_{\rm H} \Lambda_b=
  \left\{
    \begin{array}{ll}
      1+\frac{\log(3c)}{-\log b}, & \hbox{if $0<b \leq \frac{1}{3c}$;} \\
      2, & \hbox{if $\frac{1}{3c} \leq b \leq \frac{c}{2}$.}
    \end{array}
  \right.
\end{equation}

\begin{figure}[H]\label{c32}
    \centering
    \begin{subfigure}[b]{0.5\textwidth}
        \centering
        \includegraphics[width=\textwidth]{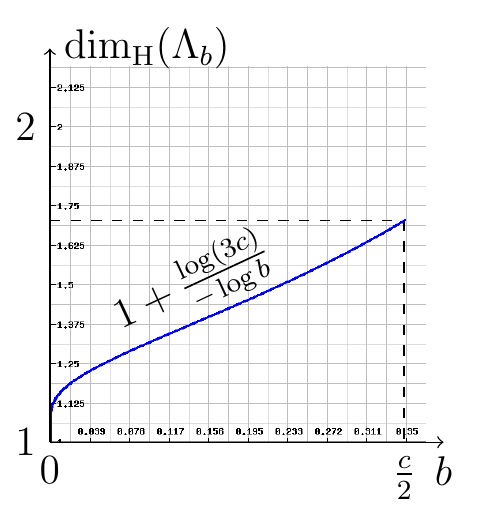}
        \caption{ $c=0.7 < \sqrt{2/3}$.}
    \end{subfigure}%
    ~
    \begin{subfigure}[b]{0.5\textwidth}
        \centering
        \includegraphics[width=\textwidth]{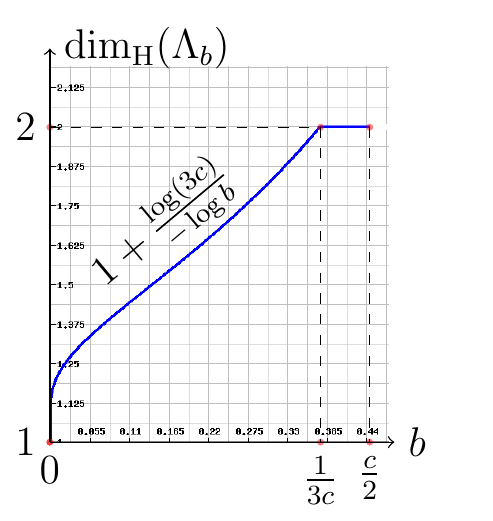}
        \caption{$c=8/9 > \sqrt{2/3}$.}
    \end{subfigure}
    \caption{In Example \ref{j48} there is no phase transition if $c<\sqrt{2/3}$ and there is a phase transition at $b=c/3$ for $\sqrt{2/3}<c<1$.}
\end{figure}
%

Which means that there is a phase transition at $b=c/3$ , where the graph of
$b\mapsto \dim_{\rm H} \Lambda_b$ is non-differentiable, although it is continuous on the interval
$b\in \left(0,\frac{c}{2}\right]$.

 \subsubsection{The definition of exceptional parameter sets $E_1,E_2, E_3$.}\label{j39}
Let $E_1$ be the set of exceptional parameters
\begin{equation}\label{j38}
  E_1:=\left\{
  c\in\left(\frac{1}{2},1\right):\mbox{ (A3)  does not hold for the IFS }
  \left\{cx,cx+1\right\}
  \right\}.
\end{equation}

Let

\begin{equation}\label{j38}
  E_2:=\left\{
  c\in\left(\frac{1}{3},1\right):\mbox{ (A3)  does not hold for the IFS }
  \left\{cx+u_1,cx+u_2,cx+u_3\right\}
  \right\}.
\end{equation}

It follows from \cite[Theorem A]{shmerkin2014absolute} that $E_1$ and $E_2$ are  sets of Hausdorff dimension zero.
 About $E_1$ note that as long as $u_1\ne u_2$ the exceptional set is the same for all IFS
 $\left\{cx+u_1,cx+u_2\right\}$.

To define the third exceptional set first we need to state the following theorem \cite[Theorem 2]{Simon_Toth}

\begin{theorem}[Simon, Tóth]\label{j44}
For a natural number $N>1$
let $\mathfrak{m}_{\lambda}^{N}$ be the self-similar measure corresponding to the  IFS $\left\{\lambda x+i\right\}_{i=0}^{N-1}$ and the uniform distribution of weights $\left\{\frac{1}{N}, \dots ,\frac{1}{N}\right\}$. Then there exists an exceptional set $R_N\subset \left(\frac{1}{N},1\right)$ having Lebesgue measure zero such that $\mathfrak{m}_{\lambda}^{N}$ is absolute continuous with $L^2$ density whenever $\lambda\in\left(\frac{1}{N},1\right)\setminus R_N$.
\end{theorem}
From this is is immediate to see that the following Corollary holds:
\begin{corollary}\label{j45}Here we use the notation of Theorem \ref{j44}. Let
  \begin{equation}\label{j37}
    E_3:=\left\{c\in \left(\frac{1}{\sqrt{3}},1\right):
  c^2\in R_3\right\}.
  \end{equation}

Then the Lebesgue measure of $E_3$ is zero and  $\mathfrak{m}_{c}^{N}$
  is absolutely continuous with continuous (consequently bounded) density for each
  $c\in\left(\frac{1}{\sqrt{3}},1\right)\setminus E_3$.
\end{corollary}

\subsection{Example for the direction-$y$ dominates case} Before we show an example for Theorem
\ref{thm:c<b}, we show a family of IFS of similarities, which satisfies Hochman's exponential separation condition.

\begin{lemma}\label{lem:Hoch}
	Let $\Psi:=\{x\mapsto\alpha x+\tau_1,x\mapsto\alpha x+\tau_2,x\mapsto\alpha x+\tau_3\}$ be an IFS on the real line such that $\alpha, \tau_1,\tau_2\in\Q$, and $\tau_3\notin\Q$. Then $\Psi$ satisfies Hochman's exponential separation condition.
\end{lemma}

The proof is similar to the proof of \cite[Theorem~1.6]{hochman2012self}.

\begin{proof}
	Without loss of generality, we may assume that $\tau_1=0, \tau_2=1$, and $0<\tau_3<1$ irrational. For any finite length word $\il\in\left\{1,2,3\right\}^{*}$, let $\pi(\il)=\sum_{k=0}^{|\il|}\tau_{i_k}\alpha^k.$
	
	By \eqref{k78}, it is enough to show that there exists $\varepsilon>0$ and a sequence $n_k$ such that for any $\il\neq\jl$ with $|\il|=|\jl|=n_k$, $|\pi(\il)-\pi(\jl)|>\varepsilon^{n_k}$. Suppose that this is not the case. That is, for every $\varepsilon>0$ there exists a $K$ such that for every $k\geq K$ there exist finite words $\il_k,\jl_k$ with $|\il_k|=|\jl_k|=k$ such that
	\begin{equation}\label{eq:contra}|\pi(\il_k)-\pi(\jl_k)|<\varepsilon^{k}.\end{equation}
	
	By definition, $\alpha=p/q$, where $p,q\in\N$ relative primes and $q\neq0$. Thus, for any $\il\neq\jl$ with $|\il|=|\jl|=k$,
	$$
	\pi(\il)-\pi(\jl)=\frac{p_1(\il,\jl)}{q^k}+\tau_3\frac{p_2(\il,\jl)}{q^k},
	$$
	where $p_1(\il,\jl),p_2(\il,\jl)\in\N$. Observe that
	$$
	|p_i(\il,\jl)|\leq(2q)^k
	$$
	for every $k\geq1$ and $\il,\jl$ with $|\il|=|\jl|=k$. On the other hand, if $p_2(\il,\jl)=0$ and $\il\neq\jl$ then $p_1(\il,\jl)\neq0$, because of the rational root test of polynomials with integer coefficients. Thus,
	$|\pi(\il)-\pi(\jl)|\geq1/q^k$.
	
	Hence, by taking $\varepsilon=1/(16q^3)$ and the sequence of words $\il_k,\jl_k$ with $|\il_k|=|\jl_k|=k$, given in the equation \eqref{eq:contra}, we get
	\begin{equation}
	\label{eq:contra2}
	\left|\frac{p_1(\il_k,\jl_k)}{p_2(\il_k,\jl_k)}-\tau_3\right|=	 \left|\pi(\il_k)-\pi(\jl_k)\right|\frac{q^{k}}{|p_2(\il_k,\jl_k)|}\leq\frac{1}{(16q^2)^{k}}.
	\end{equation}
	Therefore
	$$
	\left|\frac{p_1(\il_{k+1},\jl_{k+1})}{p_2(\il_{k+1},\jl_{k+1})}-\frac{p_1(\il_k,\jl_k)}{p_2(\il_k,\jl_k)}\right|\leq\frac{2}{(16q^2)^{k}}.
	$$
	But since the left hand side is a rational number then
	$$
	\text{either }\left|\frac{p_1(\il_{k+1},\jl_{k+1})}{p_2(\il_{k+1},\jl_{k+1})}-\frac{p_1(\il_k,\jl_k)}{p_2(\il_k,\jl_k)}\right|>\frac{1}{(2q)^{2k+1}}\text{ or }\left|\frac{p_1(\il_{k+1},\jl_{k+1})}{p_2(\il_{k+1},\jl_{k+1})}-\frac{p_1(\il_k,\jl_k)}{p_2(\il_k,\jl_k)}\right|=0.
	$$
	Thus, for every $k$ large enough $p_1(\il_k,\jl_k)/p_2(\il_k,\jl_k)\equiv P/Q$, and by \eqref{eq:contra2}, $\tau_3=P/Q$, which contradicts to the assumption that $\tau_3$ is irrational.
\end{proof}

\begin{example}\label{j33}
	Let $\mathcal{S}$ be an IFS of the form \eqref{k16}, where
	\begin{itemize}
		\item $N=3$;
		\item $c,b\in\Q$ such that $c<\min\{1/3,b\}$, $b<1$ and
		$$
		b<\min\left\{\sqrt{c},c^{1+\frac{\log3}{2\log(3c)}}\right\};
		$$
		\item $d_{i\mod3},d_{i+1\mod3}\in\Q$ and $d_{i+2\mod3}\notin\Q$ for an $i\in\N$;
		\item $u_{j\mod3},u_{j+1\mod3}\in\Q$ and $u_{j+2\mod3}\notin\Q$ for an $j\in\N$.
	\end{itemize}
	(See  Figure \ref{j34}.)
\end{example}

\begin{figure}[H]
  \centering
  \includegraphics[width=7cm]{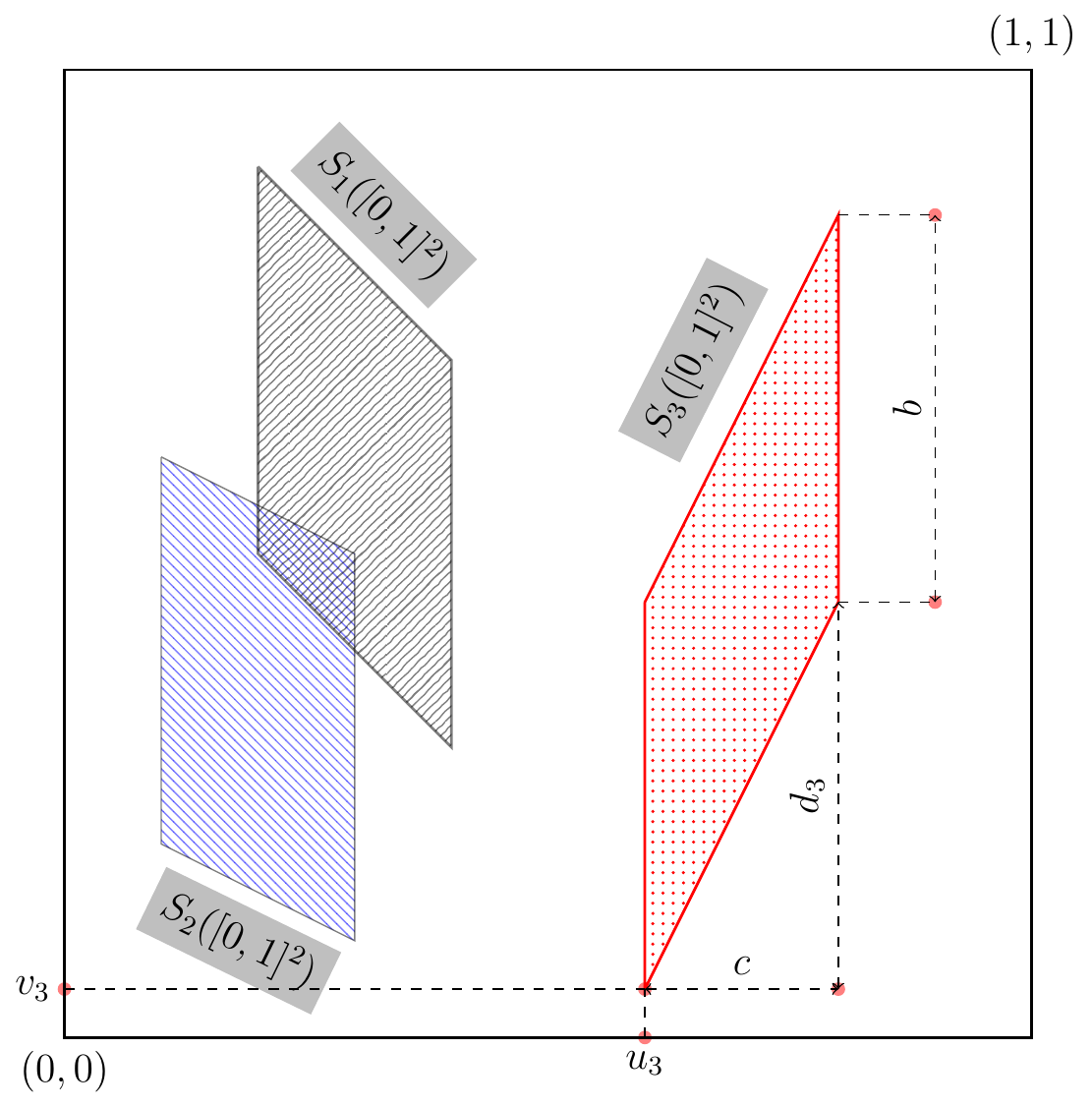}
\caption{Figure for Example \ref{j33}}\label{j34}
\end{figure}

Lemma~\ref{lem:Hoch} implies that $\mathcal{H}$ (defined in \eqref{k58}) and the backward Furstenberg IFS $\mathcal{B}$ (defined in \eqref{eq:backfurst}) satisfy Hochman's exponential separation condition. Moreover, the assumptions on $c,b$ verify that (C1), (C2) and (C5) hold in Assumption~C.

\section{Appendix}

Here we collect some consequences of the fact that a measure on the line is absolute continuous with $L^q$ density for a $q>1$.

Given a measure $\mu\ll\mathcal{L}_1$ supported by $[0,1]$ and we assume that there exists a $q>1$ such that the density $\varphi$ of $\mu$ satisfies:
  \begin{equation}\label{L1}
    C_q:=\int\limits_{0}^{1}\varphi(t)^qdt<\infty .
  \end{equation}

  \begin{lemma} \label{lem:add1}
Let $\varphi(t)$ be the density of $\nu_x$. Assume that \eqref{L1} holds.
Then for any interval $I$
\[
\mu(I) \leq C_q^{1/q} |I|^{1-1/q}.
\]
\end{lemma}
\begin{proof}
By Jensen's inequality  we have
  \begin{equation}\label{L5}
    \int\limits_{I}\varphi^q(t)dt \geq
    \Bigg(\underbrace{\int\limits_{I}^{}\varphi(t)dt}_{\mu(I)} \Bigg)^q
     \cdot |I|^{1-q}.
  \end{equation}
  This completes the proof after some obvious algebraic manipulations.
\end{proof}

\begin{definition}Given a real number $R>1$.
  We say that a not necessarily countable family $\mathcal{I}:=\left\{I_i\right\}_{i\in\Gamma}$, $I_i\subset [0,1]$ of closed  intervals is an \textbf{$R$-bad family}  if
  \begin{equation}\label{L2}
    \frac{\mu(I_i)}{|I_i|}>R,\quad \forall i\in\Gamma.
  \end{equation}
  We say that $\mathcal{I}$ is a \textbf{disjoint $R$-bad family} of intervals if $\mathcal{I}$ is an $R$-bad family and the intervals in $\mathcal{I}$ are pairwise disjoint.
\end{definition}

\begin{lemma}There exists a constant $c_1>0$ independent of everything such that for every  $R>1$ if
  $\mathcal{I}$ is an $R$-bad family of intervals
  then
  \begin{equation}\label{L3}
    \mu\left\{\bigcup\limits_{i\in\Gamma} I_i\right\}<c_1 \cdot R^{-(q-1)}.
  \end{equation}
\end{lemma}

 \begin{proof}[Proof for disjoint $R$-bad families]
  Fix an $R>1$. By Lemma \ref{lem:add1} and equation \eqref{L5}
   for every $i\in\Gamma$ we have
  \begin{equation}\label{L5}
  C_q \geq
     \mu(I_i) \cdot \left(\frac{\mu(I_i)}{|I_i|}\right)^{q-1}.
  \end{equation}
  Using \eqref{L2} and the fact that temporarily we assumed that $\mathcal{I}$ is a disjoint family
we obtain that
\begin{equation}\label{L6}
  \mu\left(\bigcup\limits_{i\in\Gamma} I_i\right) \leq  C_q \cdot R^{-(q-1)}.
\end{equation}
 \end{proof}

\begin{proof}[Proof for non-disjoint $R$-bad families]
Fix an $R>1$. For every $i\in\Gamma$ and for every $x\in I_i$ let
$$
\widetilde{I}_i(x):=\left[x-|I_i|,x+|I_i|\right]
$$

Let $A:=\bigcup\limits_i I_i$ and we form the family $\mathcal{B}$ of  closed (one dimensional) balls $\left\{\widetilde{I}_i(x)\right\}_{i,x\in I_i}$. Then by definition we can apply Besicovitch's covering theorem (using the notation of \cite[p. 30, Theorem 2.7 part (2)]{mattila1999geometry}) for $A$ and $\mathcal{B}$.  From this theorem there is a constant $Q$ independent of everything such that we can find $Q$ sub-families
$\mathcal{B}_1, \dots ,\mathcal{B}_Q$ of $\mathcal{B}$ such that

\begin{description}
  \item[(a)] for every $k$, $\mathcal{B}_k$ is a disjoint family. That is for any two distinct intervals $I,J\in B_k$ we have $I\cap J=\emptyset$.
      \item[(b)] If $J\in \mathcal{B}_k$ then
      $J=\left[x-|I_i|,x+|I_i|\right]$ for some $I_i\in\mathcal{I}$ and $x\in I_i$. Then
      \begin{equation}\label{L7}
        \frac{\mu(J)}{|J|} \geq \frac{\mu(I_i)}{2|I_i|} \geq
        \frac{1}{2}\frac{\mu(I_i)}{|I_i|} >\frac{1}{2} \cdot R
      \end{equation}
  \item[(c)] $A\subset \bigcup\limits_{k=1}^Q
  \bigcup\limits_{J\in \mathcal{B}_k}J$.

\end{description}

Using (a) and (b) for every $k$, $\mathcal{B}_k$ is a disjoint $R/2$-bad family.
Since we have already verified the assertion of the lemma for disjoint families we can apply \eqref{L6}
for $\mathcal{B}_k$ (with writing  $R/2$ instead of $R$) to get  that for all $k=1, \dots ,Q$:
\begin{equation}\label{L8}
  \mu\left\{\bigcup\limits_{J\in\mathcal{B}_k}J\right\}<  C_q \cdot 2^{q-1} \cdot R^{-(q-1)}.
\end{equation}
Using this and (c) above we obtain that
\begin{equation}\label{L9}
  \mu\left(A\right) \leq \underbrace{Q  C_q 2^{q-1}}_{c_1} \cdot R^{-(q-1)}.
\end{equation}

 \end{proof}
 From we get the following two corollaries:
\begin{corollary}\label{L10}
  Let $\mathcal{I}_n^\delta=\left\{I_i^{n,\delta}\right\}$ be a family of intervals of $|I_i^{(n,\delta)}|=c^n$ and $\mu(I_i^{(n,\delta)})>c^{n(1-\delta)}$. Then $R=c^{-n\delta}$. So,
  \begin{equation}\label{k52}
    \mu\left(\bigcup\limits_i I_i^{(n,\delta)}\right)<c_1 \cdot c^{n\delta(q-1)}.
  \end{equation}
\end{corollary}

\begin{corollary}\label{L11}Let $n \geq 1$ be arbitrary.
  Let $\pmb{\omega}\in\left\{1, \dots ,N\right\}^n$ and we write $I_{\pmb{\omega}}$ for the interval which supports
   $\Pi_x(\pmb{\omega})$ on the $x$-axis. We assume that $Nc>1$.
   Let
   $$
   L_{n,\delta}:=\left\{t\in[0,1]:
   \sum\limits_{|\pmb{\omega}|=n}\ind_{I_{\pmb{\omega}}}(t)>
   \left(Nc\right)^{n(1+\delta)}
   \right\}
   $$
   For every $t\in L_{n,\delta}$ let
   $$
   J_{t}:=\left[t-c^n,t+c^n\right].
   $$
   Then
   $$
   \frac{\mu(J_t)}{|J_t|}>\frac{1}{2}(Nc)^{n\delta}.
   $$
   That is we may apply \eqref{L3} for $\left\{J_t\right\}_{t\in{L_{n,\delta}}}$ to get that
   \begin{equation}\label{L12}
     \mu\left(L_{n,\delta}\right) \leq \mu  \left(\bigcup\limits_{t\in{L_{n,\delta}}}J_t\right)
      \leq c_12^{q-1}(Nc)^{-n\delta(q-1)}.
   \end{equation}
   recall that the constant $c_1$ was defined in \eqref{L9}, so it is independent of $\delta$.
\end{corollary}




\begin{corollary} \label{cor:add2}
Assume $\nu_x$ has $L^q$ density for all $q>1$. Then for every $K>0$ and $q>1$ there exists $C=C(K,q)$ such that for every $\ell >0$ every interval $I$ of length $Kc^\ell$ intersects at most $C (Nc)^\ell c^{-\ell/q}$ intervals $\cblue{h}_{\pmb{\tau}}([0,1]); |\pmb{\tau}| = \ell$.
\end{corollary}
\begin{proof}
Denote by $Z$ the number of intervals $\cblue{h}_{\pmb{\tau}}([0,1]); |\pmb{\tau}| = \ell$ intersecting $I$. Each of those intervals is contained in the interval $I' = B_{c^\ell}(I)$. Hence,
\[
\nu_x(I') \geq Z\cdot N^{-\ell}.
\]
Applying Lemma \ref{lem:add1} to the interval $I'$ we get the assertion.
\end{proof}

\bibliographystyle{plain}
\bibliography{affin_2016}

\end{document}